\documentclass[a4paper,11pt]{article}

\usepackage{colortbl}
\usepackage{multirow}%
\usepackage{fullpage}
\usepackage{amsmath}%
\usepackage{amsthm}%
\usepackage{amsfonts}%
\usepackage{amssymb}%
\usepackage{graphicx}
\usepackage{enumerate}
\usepackage{dsfont}
\usepackage[affil-it]{authblk}
\usepackage{natbib}
\usepackage[colorlinks=true,linkcolor=black,citecolor=blue]{hyperref}
\usepackage[titletoc,title]{appendix}

\newtheorem{theorem}{Theorem}

\newtheorem{lemma}[theorem]{Lemma}
\newtheorem{remark}{Remark}
\newtheorem{example}{Example}

\newcommand{\indep}{\perp\hspace{-.25cm}\perp}

\renewcommand{\P}{\mathbb{P}}
\newcommand{\R}{\mathbb{R}}
\newcommand{\w }[1]{\widehat{#1}}
\renewcommand{\r}{ \rightarrow }
\newcommand{\lr}{ \longrightarrow }

\DeclareMathOperator{\argmin}{argmin}
\renewcommand{\t}[1]{\text{#1}}

\begin{document}

\title{Efficiency of $Z$-estimators indexed by the objective functions}
\author{Fran\c cois Portier\footnote{Universit\'{e} catholique de Louvain,
	Institut de Statistique, Biostatistique et Sciences Actuarielles,
	Voie du Roman Pays~20,
	B-1348 Louvain-la-Neuve, Belgium.
{E-mail:} \{francois.portier\}@uclouvain.be. 
This work has been supported by Fonds de la Recherche Scientifique (FNRS) A4/5 FC 2779/2014-2017 No.\ 22342320.} 
\smallskip\\
}
\maketitle

\begin{abstract}
We study the convergence of $Z$-estimators $\w \theta(\eta)\in \R^p$ for which the objective function depends on a parameter $\eta$ that belongs to a Banach space $\mathcal H$. Our results include the uniform consistency over $\mathcal H$ and the weak convergence in the space of bounded $\R^p$-valued functions defined on $\mathcal H$. 
Furthermore when $\eta$ is a tuning parameter optimally selected at $\eta_0$, we provide conditions under which an estimated $\w \eta$ can be replaced by $\eta_0$ without affecting the asymptotic variance. Interestingly, these conditions are free from any rate of convergence of $\w \eta$ to $\eta_0$ but they require the space described by $\w \eta$ to be not too large. We highlight several applications of our results and we study in detail the case where $\eta$ is the weight function in weighted regression.  


\end{abstract}

\section{Introduction}\label{s1}

Let $P$ denote a probability measure defined on a measurable space $(\mathcal Z,\mathcal A)$ and let $(Z_1,\ldots,Z_n)$ be independent and identically distributed random elements with law $P$. Given a measurable function $f:\mathcal Z\rightarrow \mathbb{R}$, we define
\begin{align*}
Pf=\int fdP, \qquad \mathbb{P}_n f =n^{-1}\sum_{i=1}^n f(Z_i),\qquad \mathbb G_nf =n^{1/2}(\mathbb{P}_n-P)f,
\end{align*}
where $\mathbb G_n$ is called the empirical process. Considering the estimation of a Euclidean parameter $\theta_0\in \Theta\subset\mathbb{R}^p$, let $(\mathcal H,\|\cdot\|)$ denote a Banach space and let $\{\w \theta(\eta)\ : \ \eta\in \mathcal H\}$ be a collection of estimators of $\theta_0$ based on  the sample $(Z_1,\ldots,Z_n)$. Suppose furthermore that there exists $\eta_0\in \mathcal H$ such that $\w \theta(\eta_0)$ is efficient within the collection, i.e., $\w \theta(\eta_0)$ has the smallest asymptotic variance among the estimators of the collection. Such a situation arises in many fields of the statistics. For instance, $\eta$ can be the cut-off parameter in \textit{Huber robust regression}, or $\eta$ might as well be equal to the weight function in \textit{heteroscedastic regression} (see the next section for more details and examples). Unfortunately, $\eta_0$ is generally unknown since it certainly depends on the model $P$. Usually, one is restricted to first estimate $\eta_0$ by, say, $\w \eta$ and then compute the estimator $\w \theta(\w \eta)$, which should result in a not too bad approximation of $\theta_0$. It turns out that in many different situations $\w \theta(\w \eta)$ actually achieves the efficiency bound of the collection (see for instance \cite{newey1994book}, page 2164, and the reference therein, or \cite{vandervaart1998}, page 61). This is all the more surprising since the accuracy of $\w \eta$ estimating $\eta_0$ does not matter provided its consistency.

A paradigm that encompasses the previous facts can be developed \textit{via} the stochastic equicontinuity of the underlying empirical process. Define the process $\eta \mapsto \mathbb Z_n (\eta)= \sqrt n(\w \theta(\eta) - \theta_0) $ and assume that it lies in $\ell^{\infty}(\mathcal H)$, the space of bounded $\R^p$-valued functions defined on $\mathcal H$. It is stochastically equicontinuous on $\mathcal H$ if for any $\epsilon>0$,
\begin{align}\label{defrhoequicontinuity}
\lim_{\delta\r 0} \limsup _{n\r +\infty }\ P \big( \ \sup_{ \|\eta_1-\eta_2\|<\delta  } |\mathbb Z_n(\eta_1)- \mathbb Z_n(\eta_2) ) | >\epsilon \  \big) = 0,
\end{align}
where $|\cdot |$ stands for the Euclidean norm. Clearly, $\w \theta(\w \eta)$ is efficient whenever $\sqrt n (\w \theta(\w \eta)-\w\theta(\eta_0))=\mathbb Z_n(\w \eta)- \mathbb Z_n(\eta_0)$ goes to $0$ in probability. This holds true if, in addition to (\ref{defrhoequicontinuity}), we have that
\begin{align}\label{statementefficiency}
P(\w\eta \in \mathcal H)\lr 1 \qquad \t{and} \qquad \|\w  \eta- \eta_0\|\overset{P}{\lr} 0.
\end{align}
In other words, stochastic equicontinuity allows for a ``no rate" condition on $\w \eta$ given in (\ref{statementefficiency}). In fact, conditions (\ref{defrhoequicontinuity}) and (\ref{statementefficiency}) represent a trade-off we need to accomplish when selecting the norm $\|\cdot\|$. When one prefers to have $\|\cdot\|$ as weak as possible in order to prove (\ref{statementefficiency}), one needs the metric to be strong enough so that (\ref{defrhoequicontinuity}) can hold. In many statistical problems, one has
\begin{align*}
n ^{1/2}(\w \theta(\eta) - \theta_0) = \mathbb G _n \varphi _ {\eta}+ o_P(1),
\end{align*}
where the $o_P(1)$ is uniform over $\eta\in\mathcal H$ and $\varphi_{\eta}$ is a measurable function often called the \textit{influence function}. This asymptotic decomposition permits to use empirical process theory and, in particular, the Donsker property in order to show (\ref{defrhoequicontinuity}) with $\|\cdot\|$ being the $L_2(P)$-norm. As it is summarized in \cite{wellner1996}, sufficient conditions deal with the metric entropy of the class of functions $\{\varphi_\eta\ :\  \eta\in \mathcal H\}$ that needs to be small enough. 


The main purpose of the paper is to establish general conditions for the efficiency of $Z$-estimators for which the objective function depends on some $\eta\in \mathcal H$. More formally, we consider $\theta_0$ and $\w \theta(\eta)$ defined, respectively, as ``zeros" of the maps 
\begin{align*}
\theta\mapsto P\psi_\eta(\theta) \qquad \t{and}\qquad  \theta \mapsto \mathbb P_n \psi_\eta(\theta) ,
\end{align*}  
where for each $\theta\in \Theta$ and $\eta\in \mathcal H$, $\psi_\eta(\theta)$ is an $\R^p$-valued measurable map defined on $\mathcal Z$. Since for every $\eta\in \mathcal H$, $P\psi_\eta(\theta_0)=0$, we have several (possibly infinitely many) equations that characterize $\theta_0$. Hence $\eta$ can better be understood as a tuning parameter rather than as a classical nuisance parameter: $\eta$ has no effect on the consistency of $\w \theta(\eta)$ but it influences the asymptotic variance. In \cite{newey1994}, similar semiparametric estimators are studied using pathwise derivatives along sub-models and the author underlines that, for such models, ``different nonparametric estimators of the same functions should result in the same asymptotic variance" \citep[page 1356]{newey1994}. In \cite{andrews1994}, the previous statement is formally demonstrated by relying on stochastic equicontinuity, as detailed in (\ref{defrhoequicontinuity}) and (\ref{statementefficiency}). In this paper, we provide new conditions on the map $(\theta,\eta)\mapsto \psi_\eta(\theta)$ and the estimators $\w \eta$ under which the following statement holds: $\w \theta(\w \eta)$ has the same asymptotic law as $\w \theta(\eta_0)$. Despite considering slightly less general estimators than in \cite{andrews1994}, our approach alleviates the regularity conditions imposed on the map $\theta\mapsto \psi_\eta(\theta)$. They are replaced by weaker regularity conditions dealing with the map $\theta\mapsto P\psi_\eta(\theta)$. Moreover, we focus on  \textit{conditional moment restrictions models} in which $\eta$ is a weight function. In this context, our approach results in a very simple condition on the metric entropy generated by $\w \eta$, which is shown to be satisfied for classical Nadaraya-Watson estimators.


Our study is based on the weak convergence of $\{ \sqrt n  ( \w \theta(\eta)-\theta_0)\}_{\eta\in \mathcal H} $ as a stochastic process belonging to $\ell^\infty (\mathcal H) $. To the best of our knowledge, this result is new. The tools we use in the proofs are reminiscent of the $Z$-estimation literature for which we mention several relevant contributions. In the case where $\theta_0$ is Euclidean, asymptotic normality is obtained in \cite{huber1967}. Nonsmooth objective functions are considered in \cite{pollard1985}. In the case where $\theta_0$ is infinite dimensional, weak convergence is established in \cite{vandervaart1995}. The presence of a nuisance parameter with possibly, slower than root $n$ rates of convergence, is studied in \cite{newey1994}. Nonsmooth objective functions are investigated in \cite{vankeilegom2003}. Relevant summaries might be found in the books \cite{newey1994book}, \cite{wellner1996}, \cite{vandervaart1998}. 


Among the different applications we have in mind, we focus on \textit{weighted regression} for \textit{heteroscedastic models}. Even though this topic is quite well understood (see among others, \cite{robinson1987}, \cite{carroll1988} and the references therein), it represents an interesting example to apply our results. In \textit{linear regression}, given the data $(Y_i,X_i)_{i=1,\ldots,n}$, where, for each $i=1\ldots, n$, $Y_i\in \R$ is the response and $X_i\in \R^q$ are the covariates, it consists of computing
\begin{align}\label{defbetaweighted}
\w\beta(w) =\argmin_{\beta}  \sum_{i=1}^n (Y_i-\beta^TX_i ) ^2w(X_i),
\end{align}
where $w$ denotes a real valued function. Among such a collection of estimators, there exists an efficient member $\w \theta(w_0)$ (see section \ref{s31} for details). Many studies have focused on the estimation of $w_0$. For instance, \cite{carroll1982} argues that a parametric estimation of $w_0$ can be performed, and \cite{carroll1982kernel} and \cite{robinson1987} use different nonparametric estimators to approximate $w_0$.
The estimators $\w\beta (\w w)$ are shown to be efficient by relying on $U$-statistics-based decompositions. It involves relatively long and peculiar calculations that depend on both $\w w$ and the loss function.
Our approach overpass this issue by providing high-level conditions on $\w w$ that are in some ways independent from the rest of the problem. In summary we require that $\w w(x)\rightarrow w_0(x)  $ in probability, $dP(x)$-almost everywhere, and that there exists a function space $\mathcal W$ such that  
\begin{align*}
P(\w w\in \mathcal W)\r 1\qquad \text{and}\qquad\int_0^{+\infty}  \sqrt {\log\mathcal N_{[\ ]} \big(\epsilon , \mathcal W, L_r(P)\big )} d\epsilon <+\infty ,
\end{align*}
for some $r>2$, where $\mathcal N_{[\ ]}$ denotes the bracketing numbers as defined in \cite{wellner1996}. When $w_0$ is modelled parametrically, the previous conditions are fairly easy to verify. For nonparametric estimators of $w_0$, in particular for Nadaraya-Watson estimators, smoothness restrictions on $\mathcal W$ with respect to the dimension $q$ allows to obtain sufficiently sharp bounds on the bracketing number of $\mathcal W$.

The paper is organised as follows. We describe in Section \ref{sexamples} several examples of estimators for which the efficiency might follow from our approach. Section \ref{s2} contains the theoretical background of the paper. We study the consistency and the weak convergence of $\w \theta(\eta)$, $\eta\in\mathcal H$. Based on this, we obtain conditions for the efficiency of $\w \theta(\w \eta)$. At the end of Section \ref{s2}, we consider weighted estimators for conditional moment restrictions models. In section \ref{s3}, we are concerned about the estimation of the optimal weight function $w_0$ in weighted linear regression. We investigate different approaches from the parametric to the fully nonparametric. In Section \ref{s4}, we evaluate the finite sample performance of several methods by means of simulations.

\section{Examples}\label{sexamples}

As discussed in the introduction, the results of the paper allow to obtain the efficiency of estimators that depend on a tuning parameter. This occurs at different levels of statistical theory. We raise several examples in the following.

\begin{example}{(Least-square constrained estimation) }\normalfont\label{exampleLSCE}
Given $\w \theta$ an arbitrary but consistent estimator of $ \theta_0$, the estimator $\w \theta_c$ is said to be a least-square constraint estimator if it minimizes $ (\theta-\w \theta)^T \Gamma (\theta-\w \theta)$ over $\theta\in \Theta$, where $\Gamma$ is lying over the set of symmetric positive definite matrices such that $\Gamma\geq b>0$. Consequently $\w \theta_c$ depends on the choice of $\Gamma$ but since $|\w \theta_c-\w\theta|^2\leq b^{-1} | \Gamma^{1/2} (\w \theta_c-\w\theta)|^2\leq b^{-1} | \Gamma^{1/2} (\theta_0-\w\theta)|^2\lr 0$ in probability, the matrix $\Gamma$ does not affect the consistency of $\w \theta_c$ estimating $\theta_0$. It is well known that $\w\theta_c$ is efficient whenever $\Gamma$ equals the inverse of the asymptotic variance of $\w \theta$ \citep[Section 5.2]{newey1994book}. Such a class is popular among econometricians and known as \textit{minimal distance estimator}.

\end{example}

In the above illustrative example, the use of the asymptotic equicontinuity of the process $\Gamma\mapsto \sqrt n (\w \theta_c-\theta_0)$ is not really legitimate since we could obtain the efficiency using more basic tools such as the Slutsky's lemma in Euclidean space. This is due of course to the Euclideanity of $\theta$ and $\Gamma$ but also to the simplicity of the map $(\theta,\Gamma)\mapsto (\theta-\w \theta)^T \Gamma (\theta-\w \theta)$. Consequently, we highlight below more evolved examples in which either the tuning parameter is a function (Examples \ref{exampleWLS}, \ref{exampleIV} and \ref{exampleSDR}) or the dependence structure between $\theta$ and $\eta$ is complicated (Example \ref{exampleHUBER}). To our knowledge, the efficiency of the examples below is quite difficult to obtain.

\begin{example}{(weighted regression) }\normalfont\label{exampleWLS}
This includes estimators described by (\ref{defbetaweighted}) but other losses than the square function might be used to account for the distribution of the noise. Examples are $L_p$-losses, Huber robust loss (see Example \ref{exampleHUBER} for details), least-absolute deviation and quantile losses. In a general framework covering every of the latter examples, a formula of the optimal weight function is established in \cite{bates1993}. 
\end{example}

\begin{example}{(Huber cut-off) }\normalfont \label{exampleHUBER}
 Whereas weighted regression handles heteroscedasticity in the data, the cut-off in Huber regression carries out the adaptation to the distribution of the noise \citep{huber1967}. The Huber objective function is the continuous function that coincides with the identity on $[-c,c]$ ($c$ is called the cut-off) and is constant elsewhere. A $Z$-estimator based on this function permits to handle heavy tails in the distribution of the noise. The choice of the cut-off might be done according to the minimization of the asymptotic variance.
\end{example}

\begin{example}{(instrumental variable)}\normalfont \label{exampleIV}
In \cite{newey1990}, the class of \textit{nonlinear instrumental variables} is defined through the \textit{generalized method of moment}. The estimator $\w\theta$ depends on a so-called \textit{matrix of instruments} $W$, and satisfies the equation $\sum_{i=1}^n W(\widetilde Z_i)\varphi(Z_i,\theta)$ $=0 $, where each $\widetilde Z_i$ is some set of coordinates of $Z_i$ and $\varphi$ is a given function. 
A formula for the optimal matrix of instruments is available. 
 \end{example}

\begin{example}{(dimension reduction) }\normalfont \label{exampleSDR}
\cite{li1991} introduced \textit{sliced inverse regression} in which the vector $EX\psi(Y)$, when $\psi$ varies, describes a subspace of interest. A minimization of the asymptotic variance leads to an optimal $\psi_0$ that can be estimated \citep{portier2013}.
\end{example}

\section{Uniform \texorpdfstring{$Z$}{%
a }%
-estimation theory}\label{s2}

Define $(\mathcal Z_\infty, \mathcal A_\infty, P_\infty)$ as the probability space associated to the whole sequence $(Z_1,Z_2,\ldots)$. Random elements in $\ell^{\infty}(\mathcal H)$, such as $\eta\mapsto \mathbb G_n\psi_\eta(\theta)$, are not necessarily measurable. To account for this, we introduce the outer probability $P_\infty^o$  (see the introduction of \cite{wellner1996} for the definition). Each convergence: in probability or in distribution, will be stated with respect to the outer probability. 
A class of function $\mathcal F$ is said to be Glivenko-Cantelli if $\sup_{f\in \mathcal F} |(\P_n-P)f|$ goes to $0$ in $P_\infty^o$-probability. A class of function $\mathcal F$ is said to be Donsker if $\mathbb G_n f$ converges weakly in $\ell^{\infty}(\mathcal F)$ to a tight measurable element. Let $d $ denote the $L_2(P)$-distance given by $d(f,g) = \sqrt {P(f-g)^2}$. A class $\mathcal F$ is Donsker if and only if it is totally bounded with respect to $d$ and, for every $\epsilon>0$,
\begin{align}\label{defrhoequicontinuitydonsker}
\lim_{\delta\r 0} \limsup _{n\r +\infty } P_\infty^o \big( \sup_{d(f,g)<\delta  } |\mathbb G_n (f-g) | >\epsilon   \big) = 0.
\end{align}
The previous assertion follows from the characterization of tight sequences valued in the space of bounded functions \citep[Theorem 1.5.7]{wellner1996}. We refer to the book \cite{wellner1996} for a comprehensive study of the latter concepts.

For the sake of generality, we authorize the parameter of interest $\theta_0$ being a function of $\eta$. Hence we further assume that $\theta_0(\cdot)$ is an element of $\ell^{\infty}(\mathcal H)$.

\subsection{Unifrom consistency}\label{s21}

Before being possibly expressed as a $Z$-estimator, the parameter of interest $\theta_0$ is often defined as an $M$-estimator, i.e., $\theta_0\in \ell^\infty(\mathcal H)$ is such that
\begin{align}\label{formulatrueparameterMestimation}
\theta_0(\eta)= \argmin_{\theta\in \Theta} Pm_{\eta}(\theta),
\end{align}
where $m_{\eta}(\theta):\mathcal Z\r \R$ is a known real valued measurable function, for every $\theta\in \Theta$ and each $\eta \in \mathcal H$. The estimator of $\theta_0$ is noted $\w \theta$, it depends on $\eta$ since it satisfies
\begin{align}\label{formulaestimatorMestimation}
\w \theta (\eta) =\argmin _ {\theta\in \Theta} \P_n m_{\eta}( \theta).
\end{align}
Both elements $\theta_0$ and $\w\theta$ are $\R^p$-valued functions defined on $\mathcal H$. When dealing with consistency, treating $M$-estimators is more general but not more difficult than $Z$-estimators (see Remark \ref{remarkconsistencyZestimator}). This generalizes standard consistency theorems for $M$-estimators \citep[Theorem 5.7]{vandervaart1998} to uniform consistency results with respect to the objective function.

\begin{theorem}\label{thconsistencyuniform}
Assume that (\ref{formulatrueparameterMestimation}) and (\ref{formulaestimatorMestimation}) hold. If 
\begin{enumerate}[(\text{a}1)]
\item\label{condGC} $\sup_{\eta\in \mathcal H,\ \theta\in \Theta}\ |(\P_n-P)m_{\eta}(\theta)|_{}\overset{P_\infty^o}{\lr} 0$,
\item \label{condidentifiable} $\sup_{\eta\in \mathcal H}|\theta(\eta)- \theta_0(\eta)|\geq \delta >0 \quad \Rightarrow \quad \sup_{\eta\in \mathcal H}P\{m_{\eta}( \theta (\eta) )-m_{\eta}( \theta_0 (\eta) )\} \geq \epsilon >0$,
\end{enumerate} 
then we have that $\sup _{\eta\in \mathcal H} |\w \theta(\eta)-\theta_{0}(\eta)| \overset{P_\infty^o}{\lr} 0$. 
\end{theorem}

\begin{proof}
We can follow the lines of the proof of Theorem 5.7 in \cite{vandervaart1998}. Given $\delta >0$, Assumption (a\ref{condidentifiable}) implies that there exists $\epsilon>0$ such that
\begin{align*}
P_\infty^o\left(\sup_{\eta\in \mathcal H}|\w \theta(\eta)-\theta_0(\eta)|\geq \delta\right) 
&\leq P_\infty^o\left(\sup_{\eta\in \mathcal H}P\{m_{\eta}( \w \theta (\eta) )-m_{\eta}( \theta_0 (\eta) )\} \geq \epsilon \right).
\end{align*}
 By definition, $\P_n \{ m_{\eta}( \w \theta (\eta) )- m_{\eta}(  \theta_0 (\eta) ) \}\leq 0$ for every $\eta\in \mathcal H$, then we know that
\begin{align*}
&P\{m_{\eta}( \w \theta (\eta) )-m_{\eta}( \theta_0 (\eta) )\} \\
&=  (P-\P_n) m_{\eta}( \w \theta (\eta) ) + (\P_n-P) m_{\eta}( \theta_0 (\eta) ) +\P_n\{ m_{\eta}( \w \theta (\eta) )-m_{\eta}(  \theta_0 (\eta) )\}\\
&\leq (P-\P_n) m_{\eta}( \w \theta (\eta) ) + (\P_n-P) m_{\eta}( \theta_0 (\eta) )\\
& \leq 2 \sup_{\theta\in \Theta,\ \eta\in \mathcal H}|(\P_n-P) m_{\eta}( \theta  ) |,
\end{align*}
that goes to $0$ in outer probability by (a\ref{condGC}).
\end{proof}

\begin{remark}\normalfont\label{remarkglivenkocantelli}
Condition (a\ref{condGC}) requires $\{m_\eta (\theta)\ : \ \theta\in \Theta,\ \eta\in \mathcal H\}$ to be Glivenko-Cantelli. It is enough to bound the uniform covering numbers or the bracketing numbers \citep[Chapter 2.4]{wellner1996}. When $\Theta$ is unbounded, the Glivenko-Cantelli property may fail. Examples include linear regression with $L_p$-losses. In such situations, one may require the optimisation set $\Theta$ to be a compact set containing the true parameter. Another possibility is to use, if available, special features of the function $\theta\mapsto m_\eta(\theta)$ such as convexity \citep[Theorem 2.7]{newey1994}. 
\end{remark}
  
\begin{remark}\normalfont\label{remarkidentifiabilitycondition}

 Condition (a\ref{condidentifiable}) is needed for the identifiability of the parameter $\theta_0$. It says that whenever $\theta(\cdot)$ is not uniformly close to $\theta_0(\cdot)$, the objective function evaluated at $\theta(\cdot)$ is not uniformly small. Consequently, every sequence of functions $\theta_n(\cdot)$ such that $\sup_{\eta\in \mathcal H}Pm_{\eta}( \theta_n (\eta) )-m_{\eta}( \theta_0 (\eta) )\r 0 $ as $n\r+\infty$, converges uniformly to $\theta_0(\cdot)$. It is a functional version of the so called ``well-separated maximum" \citep[page 244]{kosorok2008}. It is stronger but often more convenient to verify
 \begin{enumerate}[(\text{a}1')]\label{condidentifiable2}\setcounter{enumi}{1}
\item  $\inf_{\eta\in \mathcal H} \inf_{|\theta-\theta_0(\eta)|\geq \delta} P\{m_\eta(\theta)-m_\eta(\theta_0(\eta))\}>0$,
\end{enumerate}
for every $\delta>0$. This resembles to Van der Vaart's consistency conditions in Theorem 5.9 of \cite{vandervaart1998}. To show this, suppose that $\sup_{\eta\in \mathcal H}|\theta(\eta)-\theta_0(\eta)|\geq 2 \delta$ and write 
 \begin{align*}
 P\{m_{\eta}( \theta (\eta) )-m_{\eta}( \theta_0 (\eta) )\} &\geq \mathds 1 _{\{|\theta(\eta)-\theta_0(\eta)|\geq \delta  \}} P\{m_{\eta}( \theta (\eta) )-m_{\eta}( \theta_0 (\eta) )\}\\
 &\geq 
 \mathds 1 _{\{|\theta(\eta)-\theta_0(\eta)|\geq \delta   \}}  \inf_{|\theta-\theta_0(\eta)|\geq \delta }P\{m_{\eta}( \theta )-m_{\eta}( \theta_0 (\eta) )\}\\
 &\geq \mathds 1 _{\{|\theta(\eta)-\theta_0(\eta)|\geq \delta  \}} \inf_{\eta\in \mathcal H} \inf_{|\theta-\theta_0(\eta)|\geq \delta }P\{m_{\eta}( \theta )-m_{\eta}( \theta_0 (\eta) )\}.
 \end{align*}
Conclude by taking the supremum over $\mathcal H$ in both side.
\end{remark}

\begin{remark}
\normalfont \label{remarkconsistencyZestimator}
Estimators defined through zeros of the map $\P_n\psi_{\eta}(\theta)$ are also minimizers of $|\P_n\psi_{\eta}(\theta)|$. Therefore they can be handle by Theorem \ref{thconsistencyuniform}. Let $\theta_0(\cdot)$ be such that $P\psi_{\eta}(\theta_0(\eta))=0 $ for every $\eta\in \mathcal H$. If (a\ref{condGC}) holds replacing $m$ by $\psi$ and if 
\begin{align*}
\sup_{\eta\in \mathcal H}|\theta(\eta)- \theta_0(\eta)|\geq \delta >0 \quad \Rightarrow \quad \sup_{\eta\in \mathcal H} |P\psi_{\eta}( \theta (\eta) )| \geq \epsilon >0,
\end{align*}
then the uniform convergence of zeros of $\P_n\psi_{\eta}(\theta)$ to the zero of $P\psi_{\eta}(\theta)$ holds. Given that $m_\eta$ is differentiable, $M$-estimators can be expressed as $Z$-estimators with $  \nabla_{\theta} m_\eta $ as objective function. Because any function can have several local minimums, the previous condition with $\nabla_{\theta}m_\eta$ is stronger than (a\ref{condidentifiable}). Consequently, for consistency purpose, $M$-estimators should not be expressed in terms of $Z$-estimators, see also \cite{newey1994}, page 2117.
\end{remark}

\subsection{Weak convergence}\label{s22}

We now consider the weak convergence properties of $Z$-estimators indexed by the objective functions. We assume further that $\theta_0\in \ell^{\infty}(\mathcal H)$ is such that for each $\eta\in \mathcal H$, $\theta_0(\eta)$ satisfies the $p$-dimensional set of equations
\begin{align}\label{formulatrueparameter}
P\psi_{\eta}(\theta_0(\eta))=0,
\end{align}
where $\psi_{\eta}(\theta):\mathcal Z\r \R^p$ is a known measurable function. The estimator of $\theta_0(\cdot)$ is noted $\w \theta(\cdot)$ and for each $\eta$, it is assumed that
\begin{align}\label{formulaestimator}
\P_n\psi_{\eta}(\w \theta(\eta))=0.
\end{align}
Here we shall suppose that $\sup_{\eta\in \mathcal H} |\w\theta(\eta)-\theta_0(\eta)|=o_{P_\infty^o}(1)$, so that the function $\psi_\eta$ is not intended to satisfy (a\ref{condidentifiable}). Indeed consistency may have been established from other restrictions such as minimum argument (see Remark \ref{remarkconsistencyZestimator}).

We require the (uniform) Frechet differentiability of the map $\theta\mapsto P\psi_{\eta}(\theta)$, that is, there exists $A_\eta:\Theta \mapsto  \R^{p\times p}$ such that
\begin{align}\label{eq:uniformfrechetdiff}
P\psi_{\eta}(\theta)-P\psi_{\eta}(\widetilde \theta)-A_\eta(\theta)(\theta -\widetilde \theta) = o(\theta -\widetilde \theta),
\end{align} 
where the $o(\theta -\widetilde \theta)$ does not depend on $\eta$. 

\begin{theorem}\label{theoremweakconvZestimator}

Assume that (\ref{formulatrueparameter}) and (\ref{formulaestimator}) hold. If
\begin{enumerate}[(\text{a}1)]  \setcounter{enumi}{2}
\item \label{condconsistency} $\sup_{\eta \in\mathcal H} |\w \theta(\eta)-\theta_0(\eta)|\overset{P_\infty^o}{\lr}0  $,
\item \label{condl2continuity} for every $\epsilon>0$, $\exists\delta>0$ such that $|\theta-\widetilde \theta|\leq  \delta$ implies that $\sup_{\eta\in \mathcal H}  d(\psi_\eta(\theta) ,\psi_\eta(\widetilde \theta)) \leq \epsilon $,
\item\label{conddonsker}  there exists $\delta>0$ such that the class $\Psi =\{z\mapsto  \psi_\eta(\theta)(z) \ :\  |\theta-\theta_0|<\delta,\ \eta\in \mathcal H\}$ is $P$-Donsker,
\item \label{condinvertible} the matrix $B_\eta:= A_\eta(\theta_0(\eta))$ defined in (\ref{eq:uniformfrechetdiff}), is bounded and invertible, uniformly in $\eta$,
\end{enumerate}
then we have that
\begin{align*}
 n^{1/2} (\w \theta(\eta)-\theta_0(\eta))  = - B_\eta ^{-1}\mathbb G_n \psi_{\eta}( \theta_0(\eta)) +o_{P_\infty^o}(1).
\end{align*}
Consequently, $\sqrt n (\w \theta(\eta)-\theta_0(\eta))$ converges weakly to a tight zero-mean Gaussian element in $\ell^\infty (\mathcal H)$ whose covariance function is given by $(\eta_1,\eta_2)\mapsto   B_{\eta_1}^{-1} P (\psi_{\eta_1} (\theta_0(\eta_1))\psi_{\eta_2} (\theta_0(\eta_2))^T)B_{\eta_2}^{-1}$.
\end{theorem}

\begin{proof}
We follow a standard approach by first deriving the (uniform) rates of convergence and second computing the asymptotic distribution \citep[Theorem 5.21]{vandervaart1998}. Thanks to (a\ref{condconsistency}) and (a\ref{condl2continuity}), we know that there exists a nonrandom positive sequence $\delta_n\r 0$, such that the sets 
\begin{align*}
\{\sup_{\eta \in\mathcal H} |\w \theta(\eta)-\theta_0(\eta)|<\delta_n \}\quad  \t{and}\quad  \{\sup_{\eta \in\mathcal H} d(\psi_{\eta}(\w \theta(\eta)),\psi_{\eta}(\theta(\eta))) <\delta_n \},
\end{align*}
have probability going to $1$. Without loss of generality, we assume in the proof that these events are realized. By definition of $\w\theta(\eta)$, we have
\begin{align}
o_{P_\infty^o}(1) &= n^{1/2}\{\P_n\psi_{\eta}(\w \theta) - P\psi_{\eta}(\theta_0(\eta))\}\nonumber \\
& = \mathbb G_n \{\psi_{\eta}(\w \theta(\eta))-\psi_{\eta}(\theta_0(\eta))\} +\mathbb G_n \psi_{\eta}( \theta_0(\eta)) +n^{1/2}P\{ \psi_{\eta}(\w \theta(\eta))-\psi_{\eta}( \theta_0(\eta))\}\label{equationdecompostionZestimator}.
\end{align}
The first term is treated as follows. For every $\eta\in \mathcal H$, we have that $(\psi_{\eta}(\w \theta(\eta)),\psi_{\eta}(\theta_0(\eta)))\in{\mathcal D} (\delta_n) $ with
\begin{align*}
{\mathcal D} (\delta) = \{(\psi,\widetilde \psi)\in \Psi \times \Psi\  : \   d(\psi,\widetilde \psi)< \delta\}.
\end{align*}
It follows that
\begin{align*}
 |\mathbb G_n (\psi_{\eta}(\w \theta(\eta))-\psi_{\eta}(\theta_0(\eta)))|\leq \sup_{(\psi,\widetilde \psi)\in \mathcal D( \delta_n)} |\mathbb G_n (\psi-\widetilde \psi)| .
\end{align*}  
Now using (a\ref{conddonsker}) and equation (\ref{defrhoequicontinuitydonsker}), the first term of (\ref{equationdecompostionZestimator}) goes to $0$ in $P_\infty^o$-probability. As a consequence, we know that $\mathbb G_n \psi_{\eta}( \theta_0(\eta)) +\sqrt n P\{ \psi_{\eta}(\w \theta(\eta))-\psi_{\eta}( \theta_0(\eta))\}= o_{P_\infty^o}(1)$ or, equivalently, that
\begin{align}\label{formulamestimationcentral}
\mathbb G_n \psi_{\eta}( \theta_0(\eta)) +B_\eta n^{1/2} (\w \theta(\eta)-\theta_0(\eta))
 = a_n(\eta) + o_{P_\infty^o}(1),
\end{align}
with $a_n (\eta)= -\sqrt n \{ P\{ \psi_{\eta}(\w \theta(\eta))-\psi_{\eta}( \theta_0(\eta))\}-B_\eta n^{1/2} (\w \theta(\eta)-\theta_0(\eta))\}$. Using (a\ref{condinvertible}), we have that 
\begin{align}\label{equationfrechetconsequence}
 a_n(\eta) &\leq   |n ^{1/2}(\w \theta(\eta)-\theta_0(\eta))| \sup_{|\theta_1-\theta_2|<\delta_n} \frac{|P\psi_{\eta}(\theta)-P\psi_{\eta}(\widetilde \theta)-A_\eta(\theta)(\theta -\widetilde \theta) |}{|\theta-\widetilde \theta|}\nonumber \\
& \leq o(1) \sup_{\eta\in \mathcal H} |n^{1/2} (\w \theta(\eta)-\theta_0(\eta))|.
\end{align} 
Then, by (a\ref{conddonsker}), $\sup_{\eta\in \mathcal H}|\mathbb G_n \psi_{\eta}( \theta_0(\eta))|=O_{P_\infty^o}(1)$, and using (a\ref{condinvertible}) again, in particular the full rank condition on $B_\eta$, we get
\begin{align*}
 |n^{1/2}(\w \theta(\eta) -\theta_0(\eta))|& \leq  |B_\eta n^{1/2}(\w \theta(\eta) -\theta_0(\eta))| \sup_{|u|=1} |B_\eta^{-1}u|  = O_{P_\infty^o}(1).
\end{align*}
Bringing the previous information in equation (\ref{equationfrechetconsequence}) gives that $ a_n(\eta) =o_{P_\infty^o}(1)$, therefore by equation (\ref{formulamestimationcentral}), we get
\begin{align*}
\mathbb G_n \psi_{\eta}( \theta_0(\eta)) +B_\eta n^{1/2} (\theta(\eta)-\theta_0(\eta))  = o_{P_\infty^o}(1),
\end{align*}
and the conclusion follows.
\end{proof}

\begin{remark}\normalfont \label{rk:kato}
Weak convergence of $M$-estimators is more difficult to obtain than weak convergence of $Z$-estimators. An interesting strategy is to focus on convex objective functions as developed in \cite{pollard1985}. Unlike Theorem \ref{theoremweakconvZestimator}, this approach might include non-smooth objective functions, e.g., least-absolute deviation. More recently, \cite{kato2009} considers convex objective functions that are indexed by real parameters. The main application deals with the weak convergence of the quantile regression process. 
\end{remark}

\begin{remark}\normalfont\label{remarktrueparameterindepeta}
In all the examples given in Section \ref{sexamples}, $\theta_0$ remains fixed as $\eta$ varies. Then $\psi_\eta$, $\eta\in \mathcal H$, represents a range of criterion functions available for estimating $\theta_0$. In this context, condition (a\ref{condl2continuity}) becomes 
\begin{enumerate}[(\text{a}1')]\setcounter{enumi}{3}
\item  whenever $\theta\rightarrow  \theta_0 $, $\sup_{\eta\in \mathcal H}  d(\psi_\eta(\theta) ,\psi_\eta(\theta_0)) \r 0 $,
\end{enumerate}
and condition (a\ref{condinvertible}) is reduced to 
\begin{enumerate}[(\text{a}1')]\setcounter{enumi}{5}
\item there exists $B_\eta\in \R^{p\times p}$ invertible and bounded uniformly in $\eta$ such that
\begin{align*}
P\psi_{\eta}(\theta)-P\psi_{\eta}(\theta_0)-B_\eta(\theta -\theta_0) = o(\theta -\theta_0),
\end{align*} 
\end{enumerate}
where the $o(\theta -\theta_0)$ does not depend on $\eta$.

\end{remark}

\subsection{Efficiency}\label{s24}

In this section, Theorem \ref{theoremweakconvZestimator} is used to establish conditions for the efficiency of $\w \theta (\w \eta)$ estimating $\theta_0\in \Theta$. Hence we shall assume that for every $\eta\in \mathcal H$, $\theta_0(\eta) = \theta_0$ (as in the introduction and in Remark \ref{remarktrueparameterindepeta}). Given $\w \eta$, a consistent estimator of $\eta_0$, the next theorem asserts that, whatever the accuracy of $\w \eta$, $\w \theta(\eta_0)$ and $\w \theta(\w \eta) $ have the same asymptotic variance. 

\begin{theorem}\label{thefficiency}

Assume that (\ref{formulatrueparameter}), (\ref{formulaestimator}), (a\ref{condconsistency}), (a\ref{condl2continuity}'), (a\ref{conddonsker}) and (a\ref{condinvertible}') hold. If 
\begin{enumerate}[(\t{a}1)]\setcounter{enumi}{6}
\item \label{condefficiency1} for every $\eta\in \mathcal H$, $\theta_0(\eta) = \theta_0$,
\item \label{condefficiency3} there exists $\w \eta$ and $\eta_0\in\mathcal H$ such that $P_\infty ^o (\w \eta \in \mathcal H) \lr 1$ and $\|\w\eta-\eta_0\|\overset{P_\infty^o}{\lr} 0$, 
\item \label{condefficiency4} the quantities $d(\psi_\eta(\theta_0),\psi_{\eta_0}(\theta_0))$ and $ B_\eta -B_{\eta_0}$  goes to $0$, as soon as $\|\eta-\eta_0\|\r 0$,
\end{enumerate}
then $\w \theta(\w \eta)$ has the same asymptotic law as $\w \theta (\eta_0)$. 
\end{theorem}

\begin{proof}
Write 
\begin{align*}
\w \theta(\w \eta )- \theta_0 = (\w \theta(\w \eta) - \w \theta(\eta_0) )+(\w \theta(\eta_0)  - \theta_0),
\end{align*}
we only have to show that the first term is neglectabe, i.e., $\w \theta(\w \eta) - \w \theta(\eta_0)=o_{P_\infty^o}(n^{-1/2})$. By (a\ref{condefficiency3}) and (a\ref{condefficiency4}), we can make the proof assuming that the event
\begin{align*}
\w \eta \in \mathcal H, \quad   {d(\psi_{\w \eta}(\theta_0),\psi_{\eta_0}(\theta_0))} <\delta_n,\quad  |B_{\w\eta}-B_{\eta_0}|<\delta_n,
\end{align*}
is realized, for a certain nonrandom positive sequence $\delta_n\r 0$. Then applying Theorem \ref{theoremweakconvZestimator}, we find
\begin{align*}
 n^{1/2} (\w \theta(\w \eta) - \w \theta(\eta_0) ) &=B_{\w \eta}^{-1} \mathbb G_n\{\psi_{\eta_0}(\theta_0)-\psi_{\w \eta}(\theta_0)\}+(B_{\eta_0}^{-1}-B_{\w \eta}^{-1})\mathbb G_n\psi_{\eta_0}(\theta_0) +o_{P_\infty^{o}}(1).
\end{align*}
 To obtain the convergence in probability to $0$ of the first term, we use (a\ref{conddonsker}) to rely on the stochastic equicontinuity, as in (\ref{defrhoequicontinuitydonsker}). By (a\ref{condinvertible}'), the second term equals $B_{\eta_0}^{-1}-B_{\w \eta}^{-1}= B_{\eta_0}^{-1}(B_{\w \eta}-B_{\eta_0})B_{\w \eta}^{-1} = O(\delta_n)$ times a term that is bounded in probability.
  
\end{proof}

\subsection{Conditional moment restrictions}\label{s25}

We now consider conditional moment restrictions models given by
\begin{align}\label{formulaconditionalmoment}
E(\varphi(Z,\beta_0)|X)=0, 
\end{align}
where $ X\in \mathcal X$ and $Z\in \mathcal Z$ are random variables with joint law $P$ and $\varphi$ is a known $\R^p$-valued function. The conditional restriction (\ref{formulaconditionalmoment}) implies that infinitely many (unconditional) equations are available to characterize $\beta_0$, that is, for every bounded measurable function $W$ defined on $\mathcal X$, one has
\begin{align*}
E(W(X) \varphi(Z,\beta_0)) =0.
\end{align*}
Assume that the sequence $(Z_i,X_i)_{1,\leq i\leq n}$ is independent and identically distributed from model (\ref{formulaconditionalmoment}). The estimator $\w\beta(W)$ satisfies
\begin{align}\label{formulaconditionalmomentestimator}
n^{-1}\sum_{i=1}^n  W(X_i) \varphi(Z_i,\beta) =0,
\end{align}
for every $W$ lying over the class of bounded functions denoted $\mathcal W$. Note that it includes Example \ref{exampleWLS} of Section \ref{sexamples}, for which $\mathcal W$ is a real valued class of functions, and Example \ref{exampleIV}, for which $\mathcal W$ is a $\R^{p\times p}$-valued class of functions. The following statement sheds light on special features involved by the particular objective function $(\beta,W)\mapsto  W (\cdot) \varphi(\cdot ,\beta)$. In the following, we implicitly assume enough measurability on the estimators $\w W$ in order to use Fubini's theorem freely. For this reason, outer expectations are no longer necessary. Define $e(z) = \sup_{\beta\in \mathcal B_0}|  \varphi(z,\beta)| $.

\begin{theorem}\label{thefficiencyforconditionalmomentrestriction}
Assume that (\ref{formulaconditionalmoment}) and (\ref{formulaconditionalmomentestimator}) hold. If 
\begin{enumerate}[(\text{b}1)]\setcounter{enumi}{0}
\item \label{condbidentifiability} $\sup_{W\in \mathcal W} |\w \beta(W) - \beta_0| \overset{P_\infty^o}{\lr}0 $.
\item \label{condbsmmothness}  Whenever  $\beta\r \beta_0$, $E(  \varphi(Z,\beta) -   \varphi(Z,\beta_0))^2 \lr 0$.
\item \label{condbdonsker}  Let $\delta>0$ and let  $\mathcal B_0$ denote an open ball centred at $\beta_0$. The class $\{z\mapsto   \varphi(z,\beta)\ : \ \beta \in \mathcal B_0 \}$ is $P$-Donsker and $Pe^{2+\delta}<+\infty$. Moreover, the class $e\mathcal W $ is $P$-Donsker and uniformly bounded by $W_\infty$.
\item \label{condbdifferentiability} There exists $B:\mathcal X \rightarrow \in \R^{p\times p}$ such that $ E|B(X)|<+\infty$ and, for every $x\in\mathcal X$
\begin{align*}
|E(  \varphi(Z,\beta) -  \varphi(Z,\beta_0)|X=x) -B(x) (\beta-\beta_0 )|\leq \kappa |\beta-\beta_0|^2,
\end{align*}
for some $\kappa>0$ and $EW(X)B(X)$ is bounded and invertible, uniformly in $W$.
\item \label{condboptimality}  There exist $\w W:\mathcal X\mapsto \R$ (suitably measurable) and $ W_0:\mathcal X\mapsto \R$ such that (i)  $ P_\infty  (\w W \in \mathcal W) \r 1$, (ii) $|\w W(x) - W_0(x)| \overset{P_\infty}{\lr} 0$, $P(dx)$-almost everywhere,
\end{enumerate}
 then $\w \beta(\w W)$ has the same asymptotic law as $\w \beta(W_0)$.
\end{theorem}

\begin{proof}

%
We verify each condition of Theorem \ref{thefficiency}. Clearly (b\ref{condbsmmothness}) implies (a\ref{condl2continuity}') that is enough to get (a\ref{condl2continuity}) (see Remark \ref{remarktrueparameterindepeta}). We now show that (b\ref{condbdonsker})$\Rightarrow $ (a\ref{conddonsker}). Since tightness of random vectors is equivalent to tightness of each coordinate, we can focus on each coordinate separately. In what follows, without loss of generality, we assume that $  \varphi(z,\beta)$ and $W(x)$ are real numbers, for each $\beta\in \mathcal B_0$, $z\in \mathcal Z$, $x\in \mathcal X$.  Because the class of interest is the product of two classes: $ \{z\mapsto   \varphi(z,\beta) \ : \ \beta\in \mathcal B_0\}$ and $\mathcal W$, we can apply Corollary 2.10.13 in \cite{wellner1996}. Given two pairs $(\beta,\widetilde \beta)$ and $(W,\widetilde W)$, we check that
\begin{align*}
( W(x) \varphi (z,\beta)- \widetilde W(x) \varphi (z,\widetilde \beta))^2\leq 2 (  \varphi (z,\beta)-  \varphi (z,\widetilde \beta) )^2 +2e(z)^2 (W(x)-\widetilde W(x) )^2 ,
\end{align*}
and we can easily verify every condition of the corollary. Note that the moments of order $2+\delta$ of $e$ have not been used yet. We now show that (b\ref{condbdifferentiability}) $\Rightarrow$ (a\ref{condinvertible}') by setting $B_{\eta}$ equal to $EW(X)B(X)$ which is indeed invertible and bounded. We have
\begin{align*}
E \{ W(X)(  \varphi(Z,\beta) -  \varphi(Z,\beta_0)  -B(X)(\beta-\beta_0 ))\}  \leq \kappa W_\infty |\beta-\beta_0|^2,
\end{align*}
which implies  (a\ref{condinvertible}'). Note that (\ref{formulaconditionalmoment}) implies (a\ref{condefficiency1}) and (b\ref{condboptimality}) implies (a\ref{condefficiency3}), with respect to the metric of pointwise convergence (in probability). It remains to show that (a\ref{condefficiency4}) holds with $B_{\eta}$ equal to $EW(X)B(X)$. Given $\epsilon >0$, write
\begin{align*}
|\int (\w W(x) - W_0(x)) B(x)   P(dx)  |&\leq  \int |B(x)|\ |\w W(x) - W_0(x)|  P(dx) \\
&\leq \epsilon  \int |B(x)| P(dx) + 2W_\infty \int |B(x)|\mathds 1_{\{|\w W(x) - W_0(x)|>\epsilon\}} P(dx).
\end{align*}
Taking the expectation, Fubini's Theorem leads to 
\begin{align*}
&E|\int (\w W(x) - W_0(x)) B(x)   P(dx)  | \\
&\leq \epsilon  \int |B(x)| P(dx)   + 2W_\infty  \int |B(x)| P_\infty( |\w W(x) - W_0(x)|>\epsilon ) P(dx),
\end{align*}
the right-hand side goes to $0$ by the Lebesgue dominated convergence theorem. Conclude choosing $\epsilon$ small. Using that $Ee^{2+\delta}<+\infty$, the same analysis is conducted to show that $E   \varphi(Z,\beta_0)^2 (\w W(X) - W_0(X))^2 $ goes to $0$ in $P_\infty$-probability. 

\end{proof}

\begin{remark}\label{remarkonconditionHuberloss}\normalfont

Condition (b\ref{condbdifferentiability}) deals with the regularity of the map $\beta\mapsto E(\varphi(Z,\beta)|X=x)$. This is in general weaker than asking for the regularity of the map $\beta\mapsto \varphi(z,\beta)$. For instance it permits to include the Huber loss function (defined in Example \ref{exampleHUBER}) as an example. Note also that contrary to $\mathcal W$, the class of function $\{z\mapsto   \varphi(z,\beta)\ : \ \beta \in \mathcal B_0 \}$ is not suppose to be bounded. This is important to have this flexibility in order to consider examples such as weighted least-square.
\end{remark}

\begin{remark}\label{remark:limitlaw}\normalfont
Under the conditions of Theorem \ref{thefficiencyforconditionalmomentrestriction}, the sequence $\sqrt n (\w \beta(W)-\beta_0)$ converges weakly in $\ell^\infty (\mathcal W)$ to a tight zero-mean Gaussian element whose covariance function is given by 
\begin{align*}
(W_1,W_2)\mapsto  C_{W_1}^{-1} E(   W_1(X) \varphi(Z,\beta_0)   \varphi(Z,\beta_0)^T W_2(X))C_{W_2} ^{-1},
\end{align*}
with $C_{W} =E( W(X)B(X))$.
\end{remark}

\begin{remark}\label{remark:metricW}\normalfont
In the statement of Theorem \ref{thefficiency}, the metric $\|\cdot\|$ on the space $\mathcal H$ needs to be not too strong in order to prove (a\ref{condefficiency3}) and not too weak to guarantee (a\ref{condefficiency4}). In the context of conditional moment restriction, it is required that $E\{e(Z)(W(X)- W_0(X))\}^2\rightarrow 0$ whenever $\|W-W_0\|\rightarrow 0$. If $\w W$ is a Nadaraya-Watson estimator, the uniform metric is too strong to prove (a\ref{condefficiency3}) because such an estimator can be inconsistent at the boundary of the domain. In virtue of the H\"older's inequality, one has 
\begin{align*}
E(e(Z)(W(X)- W_0(X)))^2\leq \left\{Ee(Z)^{\tfrac {2(r+2)} {r} } \right\}^{\tfrac {r} {r+2}} \left\{E(W(X)- W_0(X))^{2+r}\right\}^{\tfrac{2}{2+r}},
\end{align*}
for some $r>0$, making reasonable to work with the $L_{2+r}(P)$-metric. However the Lebesgue's dominated convergence theorem permits to use the topology of pointwise convergence (in probability). This results in weaker conditions on the moments of $e(Z)$.

\end{remark}

Finally we rely on the bracketing numbers of the associated classes. Even if the following result is weaker than Theorem \ref{thefficiencyforconditionalmomentrestriction}, the bracketing approach offers tractable conditions in practice. Moreover it allows for function classes that depend on $n$.

\begin{theorem}\label{thefficiencyforconditionalmomentrestrictionwithbracketing}
Assume that (\ref{formulaconditionalmoment}), (\ref{formulaconditionalmomentestimator}), (b\ref{condbidentifiability}), (b\ref{condbsmmothness}) and (b\ref{condbdifferentiability}) hold. If
\begin{enumerate}[(b1')]\setcounter{enumi}{2}
\item \label{condbprimedonsker}  denote by ${\Phi}= \{z\mapsto   \varphi(z,\beta)\ : \ \beta \in \mathcal B_0 \}$, then
\begin{align*}
(i)\qquad\qquad &\int _0^{+\infty} \sqrt {\log{\mathcal N_{[\ ]} \big(\epsilon, { \Phi} , L_2(P) \big)}} d\epsilon <+\infty ,\\
(ii)\qquad\qquad &\int_0^{\delta_n} \sqrt {\log{\mathcal N_{[\ ]} \big(\epsilon , \mathcal W_n , L_r(P)  \big)}} d\epsilon \longrightarrow 0 ,
\end{align*}
for every sequence $\delta_n\rightarrow 0$, and $Ee(Z)^{2+\delta}<+\infty$, for some $\delta>0$ and $r=\tfrac {2+\delta}{\delta}$,
\setcounter{enumi}{4}
\item \label{condbprimeoptimality} there exist $\w W:\mathcal X\mapsto \R$ (suitably measurable)  and $ W_0:\mathcal X\mapsto \R$ such that (i)  $ P_\infty  (\w W \in \mathcal W_n) \r 1$, (ii) $|\w W(x) - W_0(x)| \overset{P_\infty}{\lr} 0$, $P(dx)$-almost everywhere,
\end{enumerate}
in place of (b\ref{condbdonsker}) and (b\ref{condboptimality}), respectively, then $\w \beta(\w W)$ has the same asymptotic law as $\w \beta(W_0)$.
\end{theorem}

\begin{proof}
The proof follows from Theorems \ref{theoremweakconvZestimator} and \ref{thefficiency} with the following change. To obtain that
\begin{align}\label{stochequi:n}
\mathbb G_n\{  \varphi(\cdot ,\w \beta) \w W -  \varphi(\cdot ,\beta_0)  W_0   \}=o_p(1),
\end{align}
we can no longer rely on the Donsker property since the class $\mathcal W_n$ depends on $n$. Instead we use Theorem 2.2 in \cite{wellner2007},  that asserts that (\ref{stochequi:n}) holds whenever
\begin{align*}
&E(   \varphi(Z ,\w \beta) \w W(X) -  \varphi(Z,\beta_0)  W_0(X) )^2\overset{P_{\infty}^o}{\longrightarrow} 0,\\
&\int_0^{\delta_n} \sqrt {\log{\mathcal N_{[\ ]} \big(\epsilon , \Phi \mathcal W_n , L_2(P) \big)}} d\epsilon \longrightarrow 0,\\
& Pe^2=O(1),\qquad  Pe^2\mathds 1_{\{e\geq \epsilon \sqrt n \}}  \rightarrow 0,
\end{align*}
for every sequence $\delta_n\rightarrow 0$. Hence the proof follows from verifying the three previous conditions. Using the Lebesgue dominated convergence theorem and the Fubini's theorem (as it was done in the proof of Theorem \ref{thefficiencyforconditionalmomentrestriction} to obtain (a\ref{condefficiency4})), the first condition is a consequence of (b\ref{condbsmmothness}) and the pointwise convergence in probability of $\w W$. The second condition is a consequence of the bounds on the bracketing entropy of $ \Phi$ and $\mathcal W_n$ provided in (b\ref{condbprimedonsker}'). Given $\epsilon>0$, let $[\underline {  \varphi}_j,\overline{ \varphi}_j]$, $j=1,\ldots,n_1$, be brackets of $L_2(P)$-size $\epsilon$ that cover $\Phi$ and let $[ \underline W_j, \overline{W}_j]$, $j=1,\ldots,n_2$, be brackets of $L_r(P)$-size $\epsilon$ that cover $\mathcal W_n$. Because the function $z\mapsto xy$ attains its bounds on every rectangle at the edges of each rectangle, the brackets
\begin{align*}
[\min(g_{jk} )  , \max(g_{jk} )], \quad j=1,\ldots,n_1, \ k=1,\ldots ,n_2,
\end{align*}
with $g_{jk}=  (\underline {\varphi}_j \underline W_k, \underline {  \varphi}_j  \overline{W}_k , \overline{ \varphi}_j\underline W_k, \overline{ \varphi}_j \overline{W}_k )$, covers the class $\Phi\mathcal W_n$. Moreover, we have
\begin{align*}
 |\max(g_{jk} )  - \min(g_{jk} )|  &\leq | \overline {  \varphi}_j  \overline W_k  - \underline {  \varphi}_j \underline {  W}_k | \\
&\leq  W_\infty | \overline {  \varphi}_j -\underline {  \varphi}_j |   + | e| | \overline W_k  -\underline{W}_k|,
\end{align*}  
then, using Minkowski's and H\"older's inequalities, we get
\begin{align*}
\| \max(g_{jk} )  - \min(g_{jk} )\|_2\leq  \epsilon (W_\infty  + \|e\|_{2+\delta}).
\end{align*}
Hence we have shown that, for every $\epsilon>0$, $\mathcal N \big(\epsilon (W_\infty  + \|e\|_{2+\delta}) , \Phi\mathcal W_n , L_2(P) \big) $ is smaller than $\mathcal N \big(\epsilon , { \Phi} , L_2(P) \big)$ times $ \mathcal N \big(\epsilon , \mathcal W_n ,L_r(P) \big) $. This implies the integrability condition. The third condition is simply obtained by using the $2+\delta$-moments of the function $e$.
\end{proof}

\section{Application to weighted regression}\label{s3}

In this section, we are interested in the estimation of $\beta_0=(\beta_{01},\beta_{02})\in  \R^{1+q}$, defined by the following model
\begin{align}\label{regressionmodel}
E(Y|X) = \beta_{01}+\beta_{02}^TX,
\end{align}
where the conditional distribution of $Y- \beta_{01}-\beta_{02}^TX$ given $X\in \R^{q}$ is symmetric about $0$. For the sake of clarity, we focus on a linear model but under classical regularity conditions one can easily include more general link functions in our framework. 
We consider heteroscedasticity, i.e., the residual $Y - \beta_{01} - \beta_{02}^TX$ are not independent from the covariates $X_i$.  In this context, the classical least-squares estimator is not efficient and one should use weighted least-squares to improve the estimation. Assume that $(Y,X)$, $(Y_1,X_1),\ldots, (Y_n,X_n)$ are independent and identically distributed random variable form the model (\ref{regressionmodel}). A general class of estimators is given by $\w \beta(w)$, defined by
\begin{align*}
\w \beta(w) = \argmin_{(\beta_1,\beta_2)\in  \R^{1+q}} n^{-1} \sum _{i=1}^n \rho(|Y_i-\beta_1-\beta_2^T X_i|)w(X_i),
\end{align*}
where $\rho: \R^+\r\R^+$ is convex positive and differentiable and $w:\R^q\rightarrow \R$ is called the weight function. Such a class of estimators is studied in \cite{huber1967}, where a special attention is drawn on robustness properties associated to the choice of $\rho$. Note that when $\rho(x)=x^2$, we obtain the classical linear least-squares estimator, when $\rho(x)=x$, we get median regression, $\rho(x) = (x^2/2)\mathds 1_{\{0\leq x \leq c\}}+c(x-c/2) \mathds 1_{\{x>c\}}$ corresponds Huber robust regression (where $c$ needs to be chosen in a proper way). Finally quantile regression estimators and $L^p$-losses estimators are as well included in this class.

We consider three different approaches to estimate the optimal weight function $w_0$. Each approach leads to different rates of convergence. The first one is parametric, i.e., we assume that $w_0$ belongs to a class depending on an Euclidean parameter. The second one is non-parametric, i.e., we do not assume anything on $w_0$ except some regularity conditions. The third one is semiparametric and reflects a compromise between both previous approaches.

It is an exercise to verify each condition of Theorem \ref{thefficiencyforconditionalmomentrestrictionwithbracketing}. Here we focus on the special conditions dealing with the estimator $\w w$ of $w_0$, namely Conditions (b\ref{condbprimedonsker}') (ii) and (b\ref{condbprimeoptimality}'). The other conditions are classical and have been examined for different examples \citep{newey1994book}.

\subsection{Efficient weights}\label{s31}

A first question that arises is to know whether or not such a class of estimators possesses an optimal member. The answer is provided in \cite{bates1993} where the existence of a minimal variance estimator is debated. Basically, optimal members must satisfy the equation: ``the variance of the score equals the Jacobian of the expected score" (as maximum likelihood estimators). The optimal weight function is then given by
\begin{align*}
x\mapsto \frac{ 2\rho'(0) f_{\epsilon_1|X_1=x}(0) +E(g_{2,\beta_0}(Y_1,X_1)|X_1=x)}{E(g_{1,\beta_0}(Y_1,X_1)|X_1=x)},
\end{align*}
where $g_{1,\beta}(y,x)=\rho'(|y-\beta_1-\beta_2^Tx|)^2$, $g_{2,\beta}(y,x) = \rho''(|y-\beta_1-\beta_2^Tx|)$, $\epsilon _1= Y_1-\beta_{01}-\beta_{02}^TX_1 $ and $f_{\epsilon_1|X_1} $ stands for the conditional distribution of $\epsilon_1$ given $X_1$. In the following, we restrict our attention to the case $\rho'(0)=0$, so that $w_0$ simplifies to
\begin{align*}
w_0(x)= \frac{N_{\beta_0}(x)}{D_{\beta_0}(x)},
\end{align*}
with $N_{\beta}(x)= E(g_{2,\beta}(Y_1,X_1)|X_1=x)f(x)$ and $D_{\beta} = E(g_{1,\beta}(Y_1,X_1)|X_1=x)f(x)$, $f$ being the density of $X_1$.
Concerning the examples cited above, this restriction only drop out quantile regression estimators. 

A first estimator that needs to be computed is $\w \beta_{}^{(0)} = (\w \beta_{1}^{(0)},\w \beta_{2}^{(0)})$, defined as $\w \beta(w)$ with constant weight function, $w(x)=1$ for every $x\in \R^q$. Even if $\w \beta^{(0)}$ is not efficient, it is well known that it is consistent for the estimation of $\beta_0$. Since $w_0$ depends on $\beta_0$, we use $\w \beta_{}^{(0)}$ as a first-step estimator to carry on the estimation of $w_0$.

\subsection{Parametric estimation of \texorpdfstring{$w_0$}{%
a }%
}\label{s32}

In this paragraph, we assume that $w_0(x) =w(x,\gamma_0) $, where $\gamma_0$ belongs to an Euclidean space. Typically $\gamma_0$ is a vector that contains $\beta_0$. Such a situation has been extensively studied (see for instance \cite{carroll1988} and the reference therein), and it has been shown under quite general conditions that $\gamma_0$ can be estimated consistently. As a consequence we assume in the next lines that there exists $\w\gamma$ such that $\w\gamma{\rightarrow} \gamma_0$, in $P^o_\infty$-probability. The estimator of $w_0(x)$ is then given by $w(x,\w\gamma)$, for every $x\in \R^q$. As a consequence, $\beta_0$ is estimated by 
\begin{align*}
\w \beta = \argmin_{(\beta_1,\beta_2) \in \R^{1+q}} n^{-1} \sum _{i=1}^n \rho(|Y_i-\beta_1-\beta_2^T X_i|)w(X_i,\w\gamma).
\end{align*}
To verify (b\ref{condbprimedonsker}') (ii) and (b\ref{condbprimeoptimality}'), it is enough to ask the Lipschitz condition
\begin{align}\label{conditionparametriccase}
|w(x,\gamma)-w(x,\widetilde \gamma)| \leq |\gamma- \widetilde \gamma|.
\end{align}
On the one hand, (b\ref{condbprimeoptimality}') holds trivially with $\mathcal W_n$ equal to the class
$\{x\mapsto w(x,\gamma) \ :\  |\gamma-\gamma_0|\leq \delta \}$, for any $\delta>0$. On the other hand, (b\ref{condbprimedonsker}') (ii) is satisfied because the latter class has the same size as the Euclidean ball of size $\delta$. Condition (\ref{conditionparametriccase}) is sufficient but not necessary. Other interesting examples include for instance $w(x,\beta,\gamma)= (1+1 _{\{\beta_2^Tx\leq \gamma \}})^{-1}$, reminiscent of a piecewise heteroscedastic model.

In the parametric regression context given by (\ref{regressionmodel}), the parametric modelling of $w_0$ has serious drawbacks. Since a parametric form is assumed for the conditional mean, it is very restrictive in addition to parametrize the optimal weights. Moreover, the definition of $w_0$, as a quotient of conditional expectations, makes difficult to set any plausible parametric family. Finally, theorem \ref{thefficiencyforconditionalmomentrestrictionwithbracketing} does not require any rate of convergence for the estimation of $w_0$. Hence a parametric approach is unnecessary.

\subsection{Nonparametric estimation of \texorpdfstring{$w_0$}{%
a }%
}\label{s33}

We consider the bracketing entropy generated by nonparametric estimator. The classical approach taken for local polynomial estimators rely on the asymptotic smoothness of such estimators \citep{ojeda2008}. In the Nadaraya-Watson case, this smoothness approach can not succeed since whenever the support of the targeted function is compact, the bias tends to a function that jumps at the boundary of the support. In the following, the Nadaraya-Watson case is studied by splitting the bias and the variance. To compute the bracketing entropy generated by these estimators, we treat similarly and independently the numerator and the denominator. For the numerator $\w N$, we write it as $E\w N+  \Delta_N$. Since the class drawn by $E\w N$ is not random, it certainly results in a smaller entropy than the function class generated by $\Delta_N$, which turns to be included, as $n$ increases and under reasonable conditions, in a smooth class of functions.

For any differentiable function $f:\mathcal Q\subset \mathbb R^q \rightarrow \mathbb R$ and any $l=(l_1,\ldots, l_q)\in \mathbb N^q$, let $f^{(l_1,\ldots, l_q)}$ abbreviates $ \frac{\partial ^{|l|}}{\partial x_1^{l_1}\ldots \partial x_q ^{l_q}} f$, where $|l|=\sum_{j=1}^q l_j$. For $k\in \mathbb N$, $0<\alpha \leq1$, $M>0$ we say that $f\in \mathcal C_{k+\alpha ,M}(\mathcal Q)$ if, for every $ | l |\leq k$, $f^{(l)} $ exists and is bounded by $M$ on $\mathcal Q$ and, for every $ |l|=k$ and every $(x,x')\in \mathcal Q^2$, we have
\begin{align*}
|f^{(l)}(x)-f^{(l)}(x')| \leq M|x-x' |^{\alpha}.
\end{align*} 
We define the estimator by
\begin{align*}
\w w(x) = \frac{\w N(x)}{\w D(x)},
\end{align*}
where 
\begin{align*}
&\widehat {N}(x) = n^{-1} \sum_{i=1}^n  g_{2,\w \beta^{(0)} }(Y_i,X_i)  K_h(x-X_i),\\
&\widehat D(x) = n^{-1} \sum_{i=1}^n g_{1,\w \beta^{(0)} } (Y_i,X_i) K_h(x-X_i),
\end{align*}
with $K_h(\cdot) =h^{-q}K(\cdot / h)$. We require the following assumptions.

\begin{enumerate}[(c1)]
\item \label{as:consistencybetafirststep} The first step estimator is consistent, i.e., $\w \beta^{(0)}\overset{P_\infty^o}{\lr} \beta_0$.
\item \label{as:support} The variable $X_1$ has a density $f$ that is supported on a bounded convex set with nonempty interior $\mathcal Q\subset \mathbb R^q$ and there exists $b>0$ such that $\inf_{x\in \mathcal Q}  D_0 (x) \geq b>0$.
\item \label{as:smoothness}
There exists $0<\alpha_2\leq 1$, $M_1>0$ and $\mathcal B_0 \subset\R^q$, an open ball centred at $\beta_{02}$, such that for any $x\in \mathbb R^q $,  the maps $\beta\mapsto  N_{\beta}(x)$ and $\beta\mapsto  D_{\beta}(x)$ belongs to $\mathcal C_{\alpha_2,M_1}(\mathcal B_0)$. Moreover, the map $x\mapsto D_0(x)$ is uniformly continuous on $\mathcal Q$,
\item \label{as:kernel} The kernel $K:\mathbb R^q \rightarrow \mathbb R$ is a symmetric ($K(u)=K(-u)$) bounded function compactly supported such that 
\begin{align*}
\int K(u)du=1,\qquad  \int _{ \{(\mathcal Q-x)/h\} }  K(u) du\geq c>0 ,
\end{align*}
whenever $0<h\leq h_0$. Moreover, there exists $k_1\in\mathbb N$ such that for each $|l|\leq k_1+1$, the class
\begin{align*}
 \left\{ K^{(l)}\left(\frac{x-\cdot}{h}\right)\ : \ h>0, \ x\in \mathcal Q  \right\} \text{ is a bounded measurable }  VC  \text{ class,}
\end{align*}
(we use the terminology of \cite{gine2002} including the measurability requirements).
\item \label{as:bandwidth} There exists $0<\alpha_1\leq 1$ such that as $n$ increases, 
 \begin{align*}
h\rightarrow 0,\qquad \frac{|log(h)|}{nh^{q+2(k_1+\alpha_1)}}\rightarrow 0.
 \end{align*}
\end{enumerate}

Let $\mathcal V(I)$ denote the set of all the functions that take their values in the set $I$. Define 
\begin{align*}
\mathcal F_n= \mathcal A_{1,n}/ \mathcal A_{2,n}
\end{align*}
where
\begin{align*}
&\mathcal A_{1,n}= \{\mathcal C_{k_1+\alpha_1,M_1}(\mathcal Q)+\mathcal E_{N,n}\}\cap\mathcal V[-M_2,M_2], \\
&\mathcal A_{2,n}=\{ \mathcal C_{k_1+\alpha_1,M_1}(\mathcal Q)+\mathcal E_{D,n}\} \cap \mathcal V[cb/2,M_2],\\
&\mathcal E_{N,n} = \left \{ x\mapsto \int N_{\beta}(x-hu)K(u) du\ :\ \beta\in \mathcal B_0  \right\},\\
&\mathcal E_{D,n}  = \left \{ x\mapsto \int D_{\beta}(x-hu)K(u) du\ :\ \beta\in \mathcal B_0  \right\},
\end{align*}  
and $M_2= 2M_1\int |K(u)|du$.
\begin{theorem}\label{th:nonparametricclassoffunction}  If (c\ref{as:consistencybetafirststep}) to (c\ref{as:bandwidth}) hold, we have  
\begin{align*}
&P_\infty( \widehat {w} \in \mathcal F_n )\rightarrow 1,\\ 
&\log \mathcal N_{[\ ]} (\epsilon ,   \mathcal F_n  , \|\cdot\|_\infty)\leq \text{const.} \epsilon^{- q/(k_1+\alpha_1) },\qquad \text{for any } \epsilon>0, 
\end{align*}
where $\text{const.}$ depends on $q$, $\mathcal Q$, $k_1+\alpha_1$, $M_1$ and $K$.
\end{theorem}

\begin{proof}
By (c\ref{as:consistencybetafirststep}), we have that $\w \beta^{(0)}\in \mathcal B_0$ with probability going to $1$. Introducing
\begin{align*}
&\Delta_N (x) = \widehat {N}(x)-E\widehat {N}(x),\\
&\Delta_D (x) = \widehat {D}(x)-E\widehat {D}(x),
\end{align*}
we consider the following three steps. 
\begin{enumerate}[(i)]
\item \label{step1thnp} $P_\infty( \Delta_N\in \mathcal C_{k_1+\alpha_1,M_1}(\mathcal Q) )\rightarrow 1$ and $P_\infty( \Delta_D\in \mathcal C_{k_1+\alpha_1,M_1}(\mathcal Q) )\rightarrow 1$,
\item \label{step2thnp} $P_\infty( \widehat {N} \in \mathcal A_{1,n})\rightarrow 1$ and $P_\infty( \widehat {D} \in \mathcal A_{2,n} )\rightarrow 1$ (note that this is the first claim of the theorem),
\item\label{step3thnp} we compute the bound on the bracketing numbers of $\mathcal F_n$.
\end{enumerate}

\noindent \textit{Proof of (\ref{step1thnp}).} We make the proof for $ \Delta_N $ since the treatment of $\Delta_D$ is similar. Given $l=(l_1,\ldots, l_q)$ such that $|l|\leq k_1+1$, we have that $ \Delta_N^{(l)}(x)=  h^{-(q+|l|)} E[ g_{1,\w \beta^{(0)}} (Y_1,X_1)  K^{(l)} (h^{-1}(x-X_1))]$, for any $h>0$. Hence
\begin{align*}
\Delta_N^{(l)}(x)= \frac 1 {nh^{q+|l|}} \left( \sum_{i=1}^ng_{1,\w \beta^{(0)}} (Y_i,X_i) K^{(l)}(x-X_i)-E[g _{1,\w \beta^{(0)}} (Y_1,X_1)  K^{(l)}(x-X_1)] \right), 
\end{align*}
and we can apply Lemma \ref{lem:talagrand} to get that
\begin{align*}
\sup_{x\in \mathcal Q} |\Delta_N^{(l)}(x)| =O_{P_\infty} \left ( \sqrt {\frac{|\log(h)|}{nh^{q+2|l|}}} \right).
\end{align*}
Then for $1\leq |l|\leq k_1$, we know that $\Delta_N^{(l)}$ goes to $0$ uniformly over $\mathcal Q$, making the derivatives of $\Delta_N$ (with order smaller than or equal to $k_1$), bounded by $M_1$ with probability going to $1$. Now we consider the Holder property for $\Delta_N^{(l)}$ when $|l|=k_1$. For any $|x-x'|\leq h$, by the mean value theorem, we have that
\begin{align*}
|(\Delta_N^{(l)} (x)-\Delta_N^{(l)}(x')) (x-x')^{-\alpha_1}| \leq \sup_{z\in \mathcal Q} |\nabla_z \Delta_N^{(l)} (z)|\  |x-x'|^{1-\alpha_1}\leq  \sup_{z\in \mathcal Q} |\nabla_z \Delta_N^{(l)} (z)| h^{1-\alpha_1},
\end{align*}
which is, in virtue of Lemma \ref{lem:talagrand}, equal to a  $O_{P_\infty} \left ( \sqrt {\frac{|\log(h)|}{nh^{q+2(k_1+\alpha_1)}}} \right)=o_P(1)$. For any $|x-x'|>h$, we have
\begin{align*}
|(\Delta_N^{(l)} (x)-\Delta_N^{(l)}(x')) (x-x')^{-\alpha_1}| \leq 2\sup_{z\in \mathcal Q} |\Delta_N^{(l)} (z)| h^{-\alpha_1},
\end{align*}
that has the same order as the previous term. As a consequence we have shown that
\begin{align*}
\sup_{x\neq x'} |(\Delta_N^{(l)} (x)-\Delta_N^{(l)}(x')) (x-x')^{-\alpha_1}| = o_{P_\infty} (1),
\end{align*}
implying that $\Delta_N^{(l)}$ is $\alpha_1$-Holder (with constant $M_1$) with probability going to $1$.

\noindent \textit{Proof of (\ref{step2thnp}).} 
For the first statement, using (\ref{step1thnp}), it suffices to show that $\widehat {N} $ lies in $\mathcal V[-M_2,M_2] $ with probability going to $1$. Lemma \ref{lem:talagrand} and Condition (c\ref{as:smoothness}) yield 
\begin{align*}
|\widehat {N}(x)|&\leq  |E\widehat {N}(x)|  +\sup_{x\in \mathcal Q}|\widehat \Delta_N(x)|\\
&=  | \int N_{\w \beta^{(0)}}(x-hu)K(u) du|  +o_p(1)\\
&\leq M_1 \int |K(u)| du +o_p(1)\\
&\leq 2M_1 \int |K(u)| du \qquad \text{with probability going to }1.
\end{align*}  
For the second statement, it suffices to show that $\widehat {D} $ lies in $\mathcal V[cb/2,M_2] $ with probability going to $1$. To obtain the upper bound for this class we mimic what have been done before for $E\widehat {N}$. To obtain the lower bound, first write
\begin{align*}
E\widehat {D}(x) - (D_0\ast K_h)(x) 
&= \int (D_{\w \beta^{(0)}}(x-hu)-D_{ \beta_0}(x-hu))K(u) du,
\end{align*}
by Condition (c\ref{as:smoothness}), it goes to $0$ uniformly over $x\in \mathcal Q$. Then this yields
\begin{align*}
\widehat D(x) &= (D_0\ast K_h) (x) + E\widehat D(x) -  (D_0\ast K_h) (x) +\widehat \Delta_D(x)\\
&\geq  (D_0\ast K_h) (x) -\sup_{x\in \mathcal Q} |(D_0\ast K_h) (x)-  E\widehat D(x) | -\sup_{x\in \mathcal Q} |\widehat \Delta_D(x)|\\
&=  (D_0\ast K_h) (x) - o_{P_\infty} (1).
\end{align*}
Define $b(x,h)=\inf_{y\in \mathcal Q , \ |y-x|\leq hA } D_0(y) $ and $M(x,h)=\sup_{y\in \mathcal Q ,\ |y-x|\leq hA} D_0(y) $ where $A$ is such that $K(u)\mathbb I _{\{|u|> A\}}=0$ ($A$ is finite because $K$ is compactly supported), and note that, by the uniform continuity of $D_0$, $\sup_{x\in \mathcal Q} |M(x,h)-b(x,h)|\rightarrow 0 $ as $h\rightarrow 0$, it follows that
\begin{align*}
(D_0\ast K_h) (x)&=\int D_0(x+hu)  K(u) du \\
&\geq b(x,h) \int \mathbb I_{\{x+hu\in \mathcal Q \}} \{K(u)\}_+ du +M(x,h) \int \mathbb I_{\{x+hu\in \mathcal Q \}} \{K(u)\}_- du\\
&= b(x,h) \int \mathbb I_{\{x+hu\in \mathcal Q \}} K(u) du +(M(x,h)-b(x,h)) \int \mathbb I_{\{x+hu\in \mathcal Q \}} \{K(u)\}_- du\\
&\geq  b(x,h) \int \mathbb I_{\{x+hu\in \mathcal Q \}} K(u) du -o(1)\\
& \geq b \int _{ \{(\mathcal Q-x)/h\} }  K(u) du- o(1),
\end{align*}
that is greater than $c b/2>0$ whenever $n$ is large enough.

\noindent \textit{Proof of (\ref{step3thnp}).} The bound on the bracketing numbers of the associated class is obtained as follows. First, Corollary 2.7.2, page 157, in \cite{wellner1996} states that
\begin{align*}
\log \mathcal N_{[\ ]} (\epsilon ,\mathcal C_{k_1+\alpha_1 ,M_1}(\mathcal Q), \|\cdot\| _{\infty})\leq \text{const.} \epsilon^{- q/(k_1+\alpha_1) },
\end{align*}
where $\text{const.}$ depends only on $\mathcal Q$, $k_1+\alpha_1$ and $M_1$. Second by Assumption (c\ref{as:smoothness}), 
\begin{align*}
|\int N_{\beta}(x-hu) - N_{\beta'}(x-hu)  K(u) du|\leq |\beta-\beta'|^{\alpha_2} M_1 \int |K(u)| du, \\
|\int D_{\beta}(x-hu) - D_{\beta'}(x-hu)  K(u) du|\leq |\beta-\beta'|^{\alpha_2} M_1 \int |K(u)| du, 
\end{align*}
hence, the classes $\mathcal E_{N,n} $ and $\mathcal E_{D,n} $ have a $\|\cdot\|_\infty$-bracketing numbers that are smaller than the $|\cdot|$-covering numbers of $\mathcal B_0$ (up to some multiplicative constant) that are obviously smaller than the $\|\cdot\|_\infty$-bracketing numbers of $\mathcal C_{k_1+\alpha_1 ,M_1}(\mathcal Q)$. Since the class $ \mathcal A_{1,n}$ (resp. $\mathcal A_{2,n}$) coincides with the set $\mathcal C_{k_1+\alpha_1 ,M_1}(\mathcal Q)$ plus $\mathcal E_{N,n} $ (resp. $\mathcal E_{D,n} $), the bracketing numbers of $ \mathcal A_{1,n}$ (resp. $ \mathcal A_{2,n}$) are then smaller than the square of the bracketing number of $\mathcal C_{k_1+\alpha_1 ,M_1}(\mathcal Q)$, i.e., 
\begin{align*}
\log \mathcal N_{[\ ]} (\epsilon ,\mathcal A_{j,n}, \|\cdot\|_\infty)\leq \text{const.} \epsilon^{-q/(k_1+\alpha_1) },
\end{align*}
for $j=1,2$. Let $[\underline{N}_1,\overline N_1],\ldots,[\underline{N}_{n_1},\overline N_{n_1}]$ (resp.  $[\underline{D}_1,\overline D_1],\ldots,[\underline{D}_{n_2},\overline D_{n_2}]$) be $\|\cdot\|_\infty$-brackets of size $\epsilon$ that cover $ \mathcal A_{1,n}$ (resp. $ \mathcal A_{2,n}$). Clearly, by taking $\underline D_1\vee b,\ldots, \underline D_{n_2}\vee b$ in place of $\underline D_1,\ldots,\underline D_{n_2}$, we can assume that the elements of the brackets of $\mathcal A_{2,n}$ are larger than or equal to $b$. By a similar argument, every brackets of $\mathcal A_{1,n}$ are bounded by $M_2$. Then for any $N\in \mathcal A_{1,n}$ and $D\in \mathcal A_{2,n}$, there exists $1\leq i\leq n_1$ and $1\leq j\leq n_2$ such that
\begin{align*}
&\frac{\underline N_i }{\overline D_j }\leq \frac{N}{D }\leq \frac{\overline N_i }{\underline D_j},\\
&\left\|\frac{\underline N_i }{\overline D_j }-\frac{\overline N_i }{\underline D_j} \right\|_\infty \leq \text{const.} \epsilon,
\end{align*}
where $\text{const.}$ is a constant that depends only on $b$ and $M_2$. As a consequence we have exhibited an $\|\cdot\|_\infty$-bracketing of size $\text{const.}\epsilon$ with $n_1n_2$ elements, yielding to the statement of the theorem.


\end{proof}

\begin{remark}\normalfont
On the one hand, no strong assumptions are imposed on the regularity of the targeted functions $N_0$ and $D_0$. Actually, we only require the uniform continuity of $D_0$ to hold. On the other hand, the kernel needs to be many times differentiable. Hence it consists of approximating a function, non necessarily regular, by a smooth function. In this way, we can control the entropy generated by the class of estimated function.
\end{remark}

\begin{remark}\normalfont
Another way to proceed is to consider the classes
\begin{align*}
\mathcal E_{N} = \left \{ x\mapsto \int N_{\beta}(x-hu)K(u) du\ :\ \beta\in \mathcal B_0,\ h>0  \right\},\\
\mathcal E_{D}  = \left \{ x\mapsto \int D_{\beta}(x-hu)K(u) du\ :\ \beta\in \mathcal B_0,\ h>0  \right\},
\end{align*}
in place of $\mathcal E_{N,n}$ and $\mathcal E_{D,n}$. The resulting classes are larger but they no longer depend on $n$. To calculate the bracketing entropy of the spaces $\mathcal E_{N}$ and $\mathcal E_{D}$, one might consider the $L_r(P)$-metric rather than the uniform metric because the latter involves some difficulties at the boundary points.
\end{remark}

\begin{remark}\normalfont
The assumptions on the Kernel might seem awkward in the first place. Examples of kernels that satisfy the last requirement in  (c\ref{as:kernel}) are given in \cite{nolan1987} (Lemma 22), see also \cite{gine2002}. An interesting fact is that  $\{ \widetilde K\left(\frac{x-\cdot}{h}\right)\ : \ h>0, \ x\in \mathbb R  \}$ is a uniformly bounded VC class of function, whenever $\widetilde K$ has bounded variations. The assumption that $\int _{ \{(\mathcal Q-x)/h\} }  K(u) du\geq c>0$ holds true if $\mathcal Q$ is a smooth surface, i.e., the distance between $x\in \R^q$ and $Q$ is a differentiable function of $x$.
Note also that in the one-dimensional case, it is always verified. Moreover, this condition permits to include the case of non-smooth surfaces such as cubes.
\end{remark}


\subsection{Semiparametric estimation of \texorpdfstring{$w_0$}{%
a }%
}\label{s34}

The nonparametric approach involves a smoothing in the space $\mathbb R^q$. It is well known that the smaller the dimension $q$, the better the estimation. Although it does not affect the the efficiency of the estimators (any rate of convergence of $\w w$ to $w_0$ satisfies condition (b\ref{condbprimeoptimality}')), it certainly influence the small sample size performance.

There exist different ways to introduce a semiparametric procedure to estimate $w_0$. We rely on the single index model. In our initial regression model (\ref{regressionmodel}), the conditional mean of $Y$ given $X$ depends only on $\beta_{02}^T X$. Given this, it is slightly stronger to ask that the conditional law of $Y$ given $X$ is equal to the conditional law of $Y$ given $\beta_{02}^TX$, in other words, that
\begin{align}\label{condindepconditional}
Y\indep X| \beta_{02}^TX,
\end{align}
or equivalently that,
\begin{align*}
E(g(Y)|X) = E(g(Y)|\beta_{02}^TX),
\end{align*}
for every bounded measurable function $g$. Such an assumption has been introduced in \cite{li1991} to estimate the law of $Y$ given $X$. Here (\ref{condindepconditional}) is introduced in a different spirit than in \cite{li1991}: since we have already assumed a parametric regression model in (\ref{regressionmodel}), condition (\ref{condindepconditional}) is an additional mild requirement, that serves only the estimation of $w_0$. The calculation of semiparametric estimators of $w_0$ is done by using similar tools as in the previous section. In order to fully benefit from Condition (\ref{condindepconditional}), we realize the smoothing in a low-dimensional subspace of $\R^q$. We define the estimator by
\begin{align*}
\w w(x) = \frac{\w N_{ \w \beta^{(0)} }(\widehat \beta^{(0)T}_2x)}{\w D_{\w \beta^{(0)}}(\widehat \beta^{(0)T}_2 x)},
\end{align*}
where 
\begin{align*}
&\widehat {N}_{ \beta}(t) = n^{-1} \sum_{i=1}^n  g_{2, \beta}(Y_i,X_i)  L_h(t-\beta_2^TX_i),\\
&\widehat D_{ \beta }(t) = n^{-1} \sum_{i=1}^n g_{1, \beta}(Y_i,X_i) L_h(t-\beta_2^TX_i),
\end{align*}
with $ L_h(\cdot) =h^{-1}L(\cdot / h)$. The proofs are more involved than in the nonparametric case because of the randomness of the space generated by $\w \beta^{(0)}_2$ in which the smoothing is realized .

\section{Simulations}\label{s4}

The asymptotic analysis conducted in the previous sections demonstrate that the estimation of $w_0$ does not matter provided its consistency, e.g., $\w w_1$ and $\w w_2$ can converge to $w_0$ with different rates, while $\w \beta(\w w_1)$ and $\w \beta(\w w_2)$ are asymptotically equivalent. Nevertheless when the sample size is not very large, differences might arise between the procedures. In the next we consider the three approaches investigated in the latter section, namely parametric, nonparametric and semiparametric. Each of these procedures results in different rates of convergence of $\w w$ to $w_0$. Here our purpose is two folds. First to provide a clear picture of the small sample size performances of each method, second, from a practical point of view, to analyse the relaxation of the smoothness properties of $w_0$.

We consider the following heteroscedastic linear regression model. Let the sequence $(X_i,Y_i)_{i=1,\ldots,n}$ be independent and identically distributed such that
\begin{align*}
Y_1= \beta_{01}+ \beta_{02}^TX_1+\sigma_0(X_1) \epsilon_1,
\end{align*}
where $(X_1,e_1)\in \R^{q+1}$ has a standard normal distribution and $(\beta_{01},\beta_{02})= (1,\ldots, 1)/\sqrt {q+1}$. The weighted least-square estimator is given by
\begin{align}\label{eq:betaWLS}
\w \beta(w) = \w \Sigma(w)^{-1}\w \gamma(w),
\end{align}
with $\w \Sigma(w)=n^{-1}\sum_{i=1}^n \widetilde X_i\widetilde X_i^T w(X_i) $, $\w \gamma(w) =n^{-1}\sum_{i=1}^n \widetilde X_iY_iw(X_i)$ and $\widetilde X_i^T= (1,X_i^T)$, for $i=1,\ldots, n$. Let $ \w \beta^{(0)}=(\w \beta^{(0)}_1,\w \beta^{(0)}_2) $, with $\w \beta^{(0)}_1\in \R$ and $\w \beta^{(0)}_2\in \R^q$, denote the first step estimator with constant weights. For different sample sizes $n$ and also several dimensions $q$, we consider two values for $\sigma_0$: one smooth function and one discontinuous function, given by, respectively,
\begin{align*}
\sigma_{01}(x) =\frac{\beta_{02}^Tx}{|\beta_{02}|} \qquad \t{and}\qquad \sigma_{02}(x) = \frac 1 2 +2\cdot \mathds 1_{\{\beta_{02}^Tx>0\}}.
\end{align*}
%
In each case, the optimal weight function $w_0=1/\sigma_{0k}^{2}$ is estimated by these methods:
\begin{enumerate}[(i)]
\item parametric, $\w w$ is computed using $\w \beta^{(0)} $ in place of $\beta_0$ in the formula of $w_0$, 
\item nonparametric, applying a kernel smoothing procedure (of the Nadaraya-Watson type), $\w w$ is given by
\begin{align*}
\frac{\sum_{i=1}^n  K_h(X_i-x)}{\sum_{i=1}^n (Y_i -\w \beta_{1}^{(0)}  -\w \beta_2^{(0)T}  X_i )^2 K_h(X_i-x)},
\end{align*}
\item semiparametric, applying a kernel smoothing procedure in a reduced sample-based space, $\w w$ is given by
\begin{align*}
\frac{ \sum_{i=1}^n K_h(\w A(X_i-x))}{ \sum_{i=1}^n (Y_i -\w \beta_{1}^{(0)}  -\w \beta_2^{(0)T} X_i )^2 K_h(\w A(X_i-x))} ,
\end{align*}
where $\w A = \w P^{(0)}_2+ \epsilon  I $ and $\w P^{(0)}_2$ denotes the orthogonal projector onto the linear space generated by $\w \beta^{(0)}_2$.
\end{enumerate}
For (ii) and (iii), the kernel $K$ is the Epanechnikov kernel given by $K(u) = c_q (1-|u|^2)_+ $, where $c_q$ is such that $\int K(u) du=1$. Whenever $\w w$ is computed according to one of the method (i), (ii) or (iii), the final estimator of $\beta_0$ is computed with $\w \beta(\w w)$ given by (\ref{eq:betaWLS}).

As highlighted in the previous section, the bandwidth $h$ needs to be chosen in such a way that $\w w$ is smooth. In practice, we find that choosing $h$ by cross validation is reasonable. More precisely, we consider the estimation of $\sigma_0^2(x)$ by $\w \sigma^2(x)$, i.e.,
\begin{align*}
h_{cv} = \argmin_{h>0} \sum_{i=1}^n (  (Y_i-\w \beta_{1}^{(0)}  -\w \beta_2^{(0)T} X_i)^2 - \w \sigma^{2(-i)}(X_i) )^2,
\end{align*}
where $ \w \sigma^{2(-i)}(x)$ is either the leave-one-out nonparametric estimator of $\sigma_0^2(x)$ given by (ii) or the leave-one-out semiparametric estimator of $\sigma_0^2(x)$ given by (iii). Such a choice of $h$ has the advantage to select automatically the bandwidth without having to take care of the respective dimensionality of each procedure semi- and nonparametric. In every examples, the semiparametric $h_{cv}$ was smaller than the nonparametric $h_{cv}$.

For the semiparametric method, the matrix $\w A$ denotes the orthogonal projector onto the space generated by $\w \beta_0$ perturbed by $\epsilon$ in the diagonal. This permits not to have a blind confidence in the first estimator $\w\beta_0$ accounting for variations of $w_0$ in the other directions. Hence $\epsilon$ is reasonably selected if $\epsilon I$ has the same order as the error $\w P_{2}^{(0)}-P^{(0)}_2$, where $P^{(0)}_2$ is the orthogonal projector onto the linear space generated by  $\beta_2^{(0)}$. Computed with the Frobenius norm $|\cdot |_F$, it yields
\begin{align*}
|\w P_{2}^{(0)}- P_{2}^{(0)}|_F^2 = {2  \text{trace}( (I- P_{2}^{(0)}) \w P_{2}^{(0)} ) }{}= \frac{2  | (I- P_{2}^{(0)}) \w  \beta_2^{(0)} |^2 }{|\w \beta^{(0)}_2|^2}.
\end{align*}
The numerator is approximated by an estimator of the average value of its asymptotic law (in the case where $\epsilon \indep X$), it gives $2\w \sigma^2n^{-1/2} \sum_{k=1}^q \w \lambda_k^2  $ where $\w \lambda_k$ are the eigenvalues associated to the matrix $(I-\w P_{2}^{(0)})\w \Sigma^{-1}_{2}  (I-\w P_{2}^{(0)}) $, $\w \sigma^2 = n^{-1}\sum_{i=1}^n (Y_i-\w \beta_0-\w\beta_0^T X_i)^2$ and $\Sigma^{-1}_{2}$ denotes the $q\times q$ lower triangular block of the inverse of $n^{-1} \sum_{i=1}^n \widetilde X_i\widetilde X_i^T$. As a consequence $\epsilon$ is given by
\begin{align*}
\epsilon = \sqrt{\frac{2\w \sigma^2 \sum_{k=1}^q \w \lambda_k^2}{nq|\w \beta|^2 }}  ,
\end{align*}
where $\sqrt q$ appears as a normalizing constant.

\begin{figure}\centering
\includegraphics[width=5.2cm,height=8.5cm]{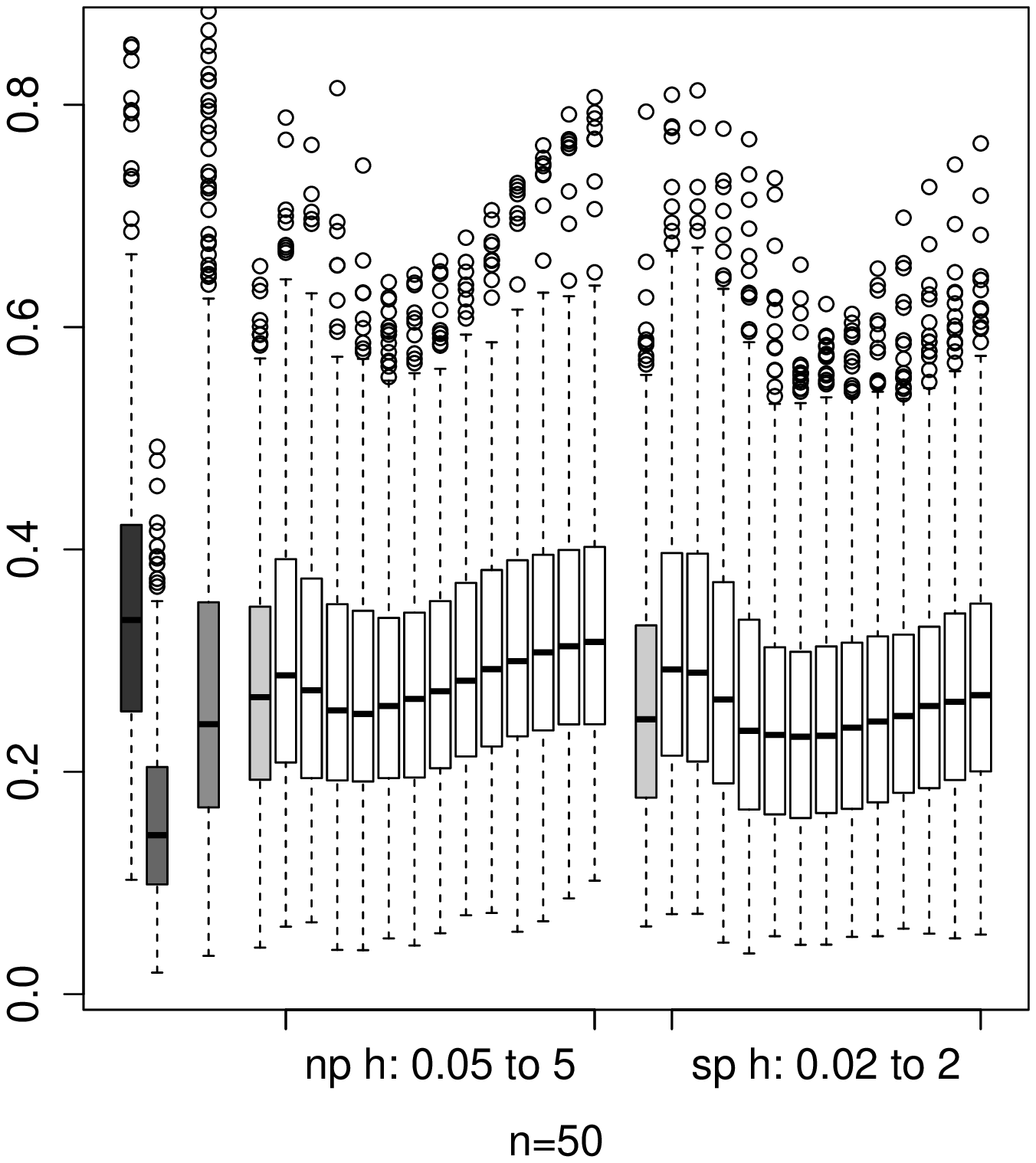} 
\includegraphics[width=5.2cm,height=8.5cm]{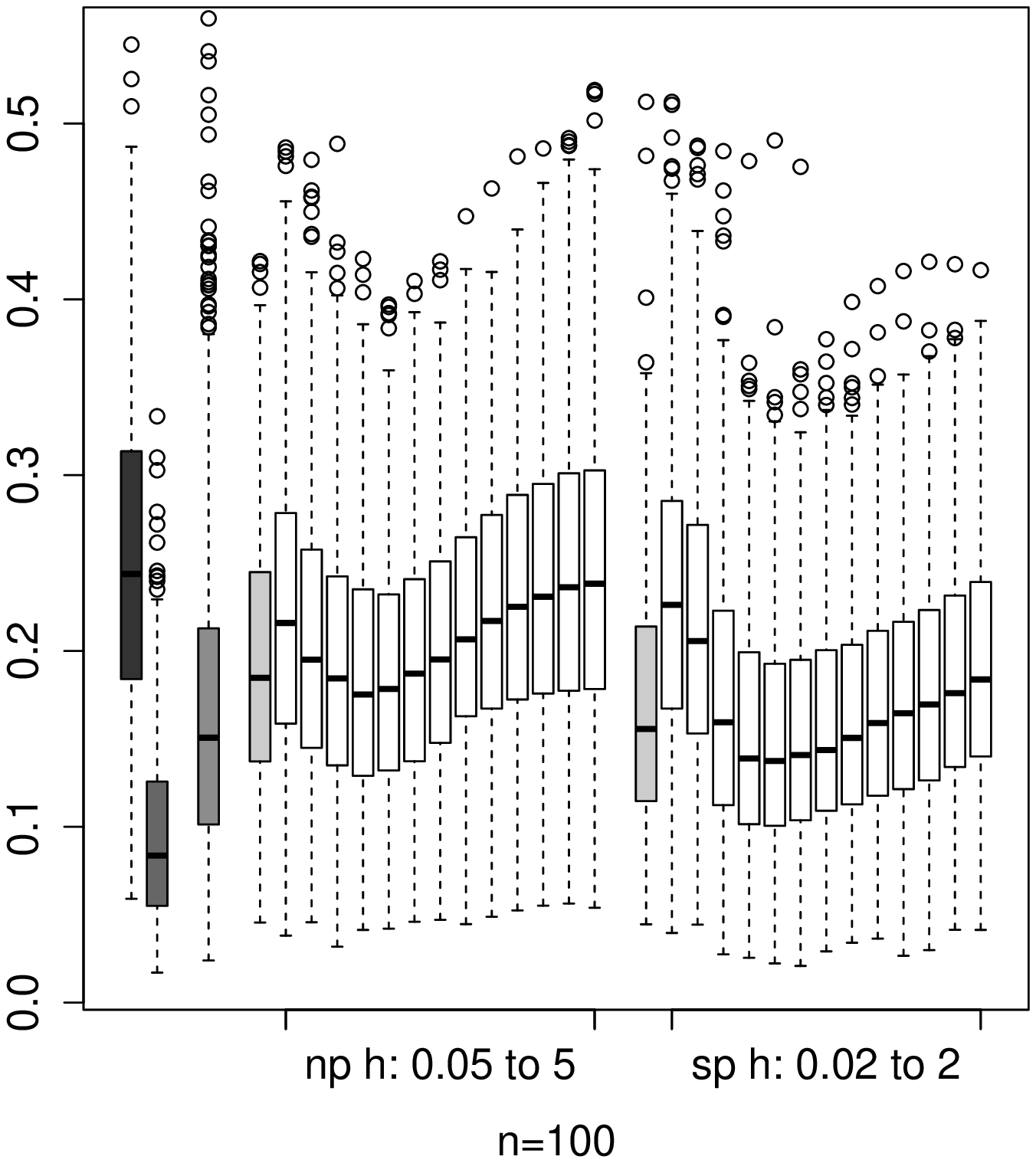} 
\includegraphics[width=5.2cm,height=8.5cm]{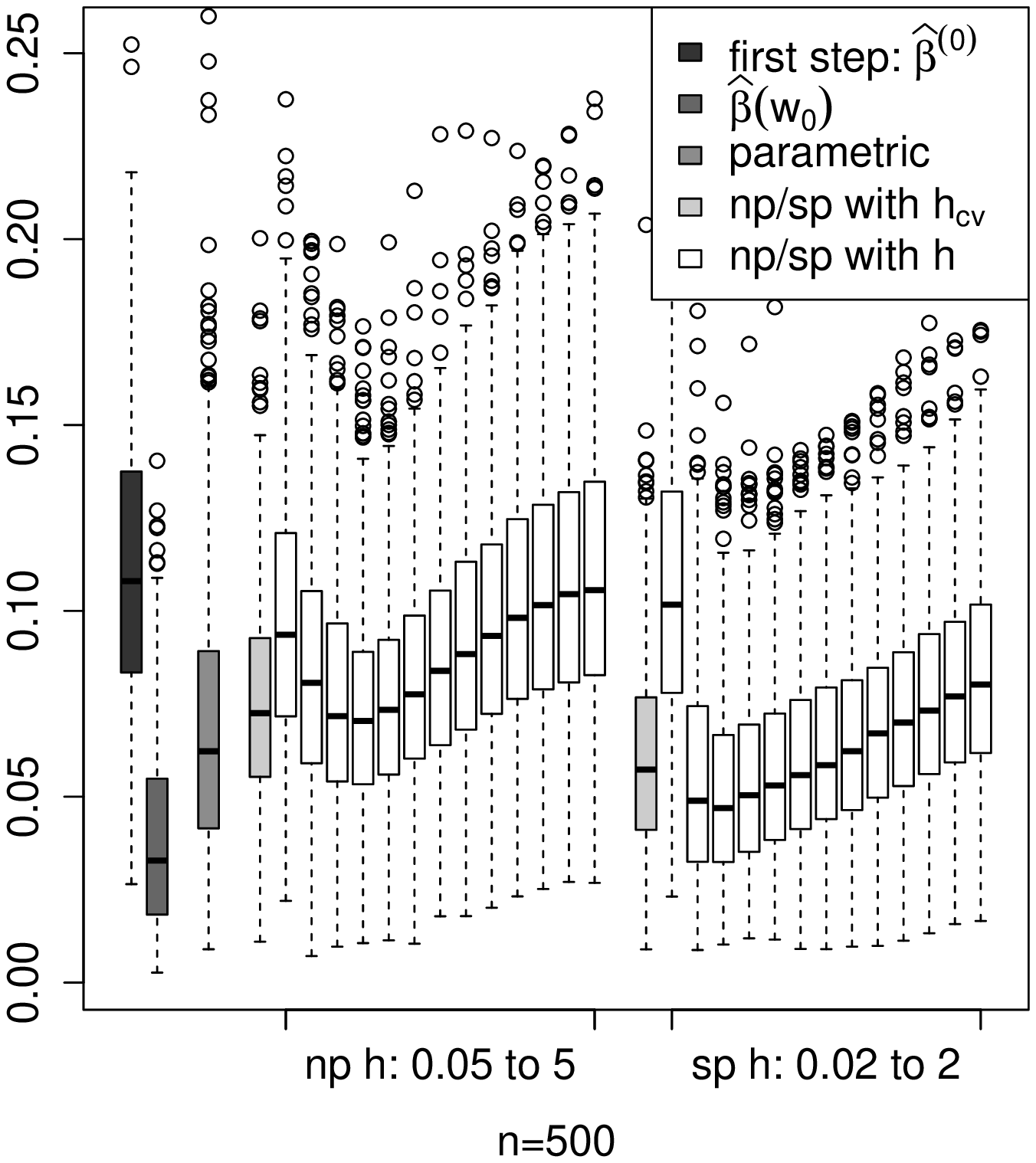} 
\caption{Each boxplot is based on $500$ estimates of $|\w \beta-\beta_0|^2$ when $q=4$ and $\sigma_{0}(x) =\frac{\beta_{02}^Tx}{|\beta_{02}|}$. The parametric, nonparametric (np) and semiparametric (sp) approaches are respectively based on (i), (ii) and (iii).}
\label{simu1}
\end{figure}

\begin{figure}\centering
\includegraphics[width=5.2cm,height=8.5cm]{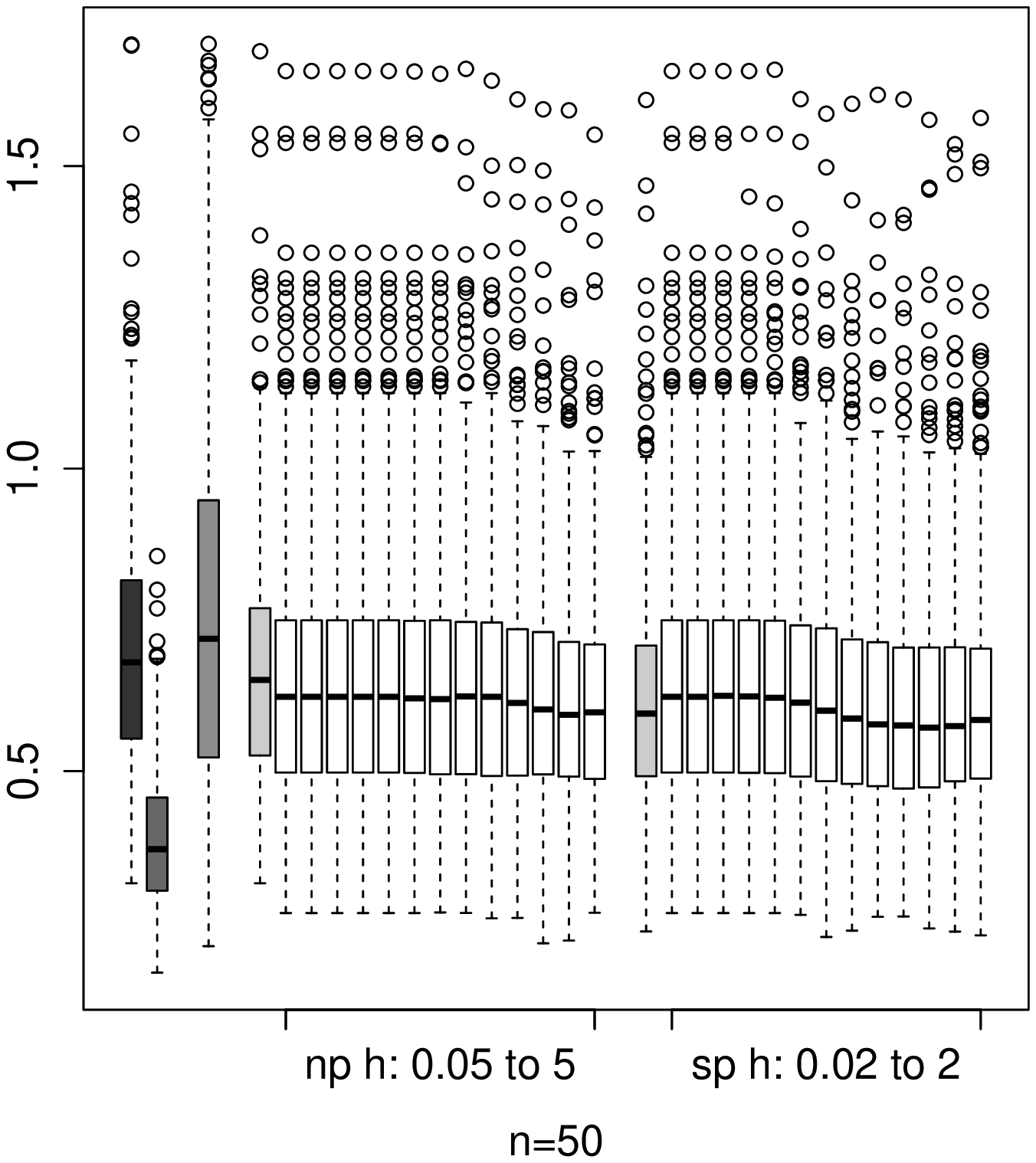} 
\includegraphics[width=5.2cm,height=8.5cm]{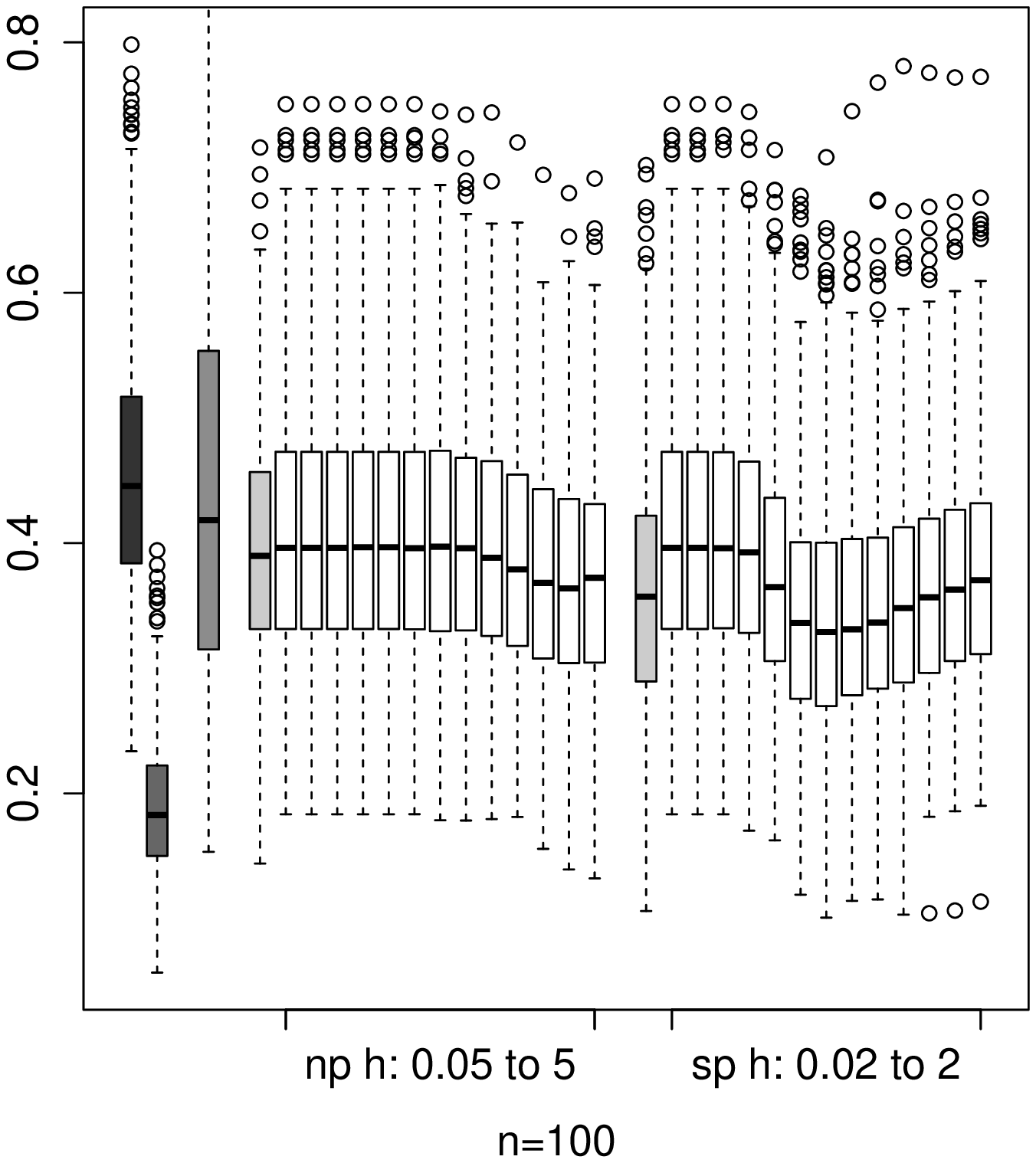} 
\includegraphics[width=5.2cm,height=8.5cm]{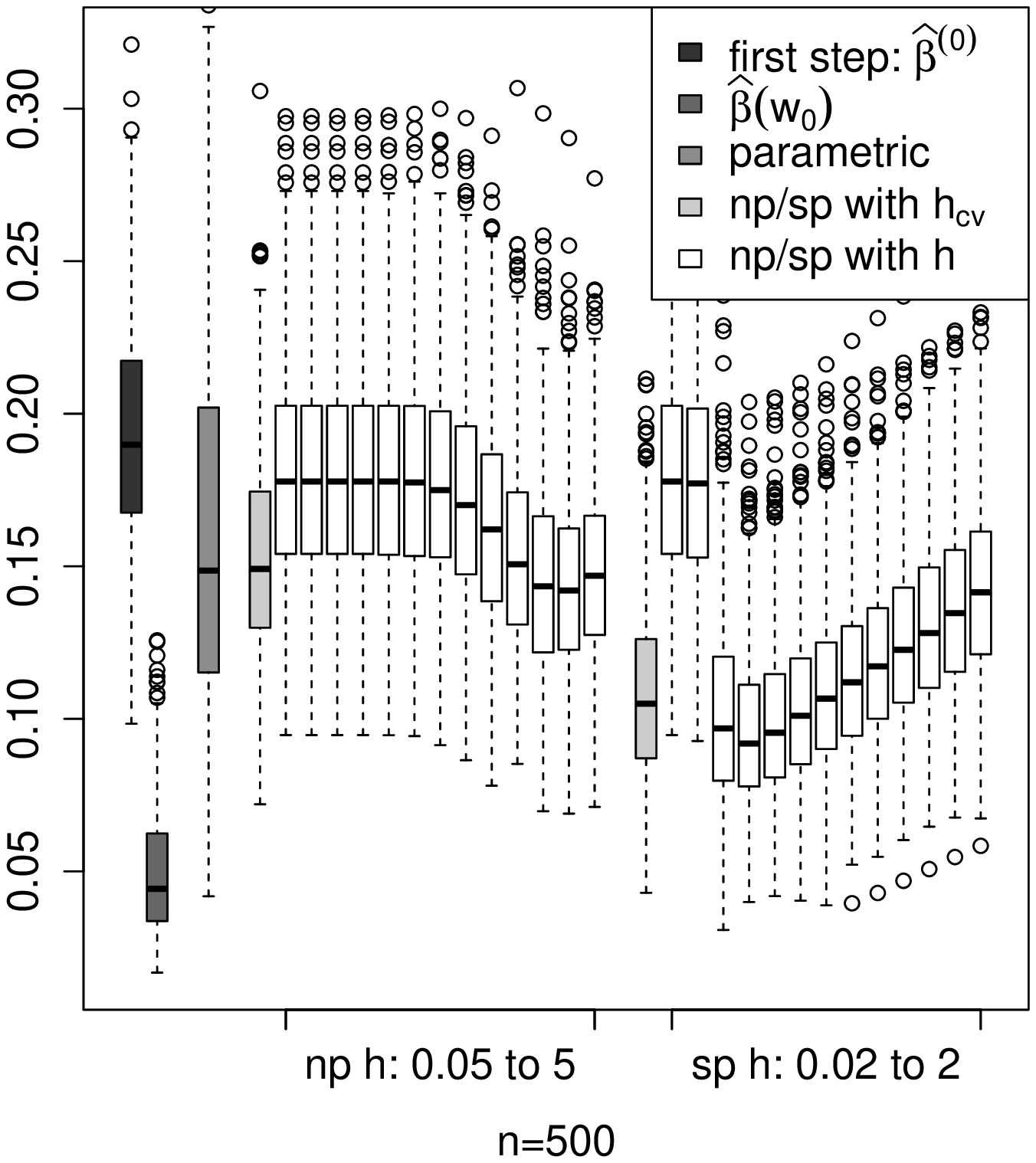} 
\caption{Each boxplot is based on $500$ estimates of $|\w \beta-\beta_0|^2$ when $q=16$ and $\sigma_{0}(x) =\frac{\beta_{02}^Tx}{|\beta_{02}|}$. The parametric, nonparametric (np) and semiparametric (sp) approaches are respectively based on (i), (ii) and (iii).}
\label{simu2}
\end{figure}

\begin{figure}\centering
\includegraphics[width=5.2cm,height=8.5cm]{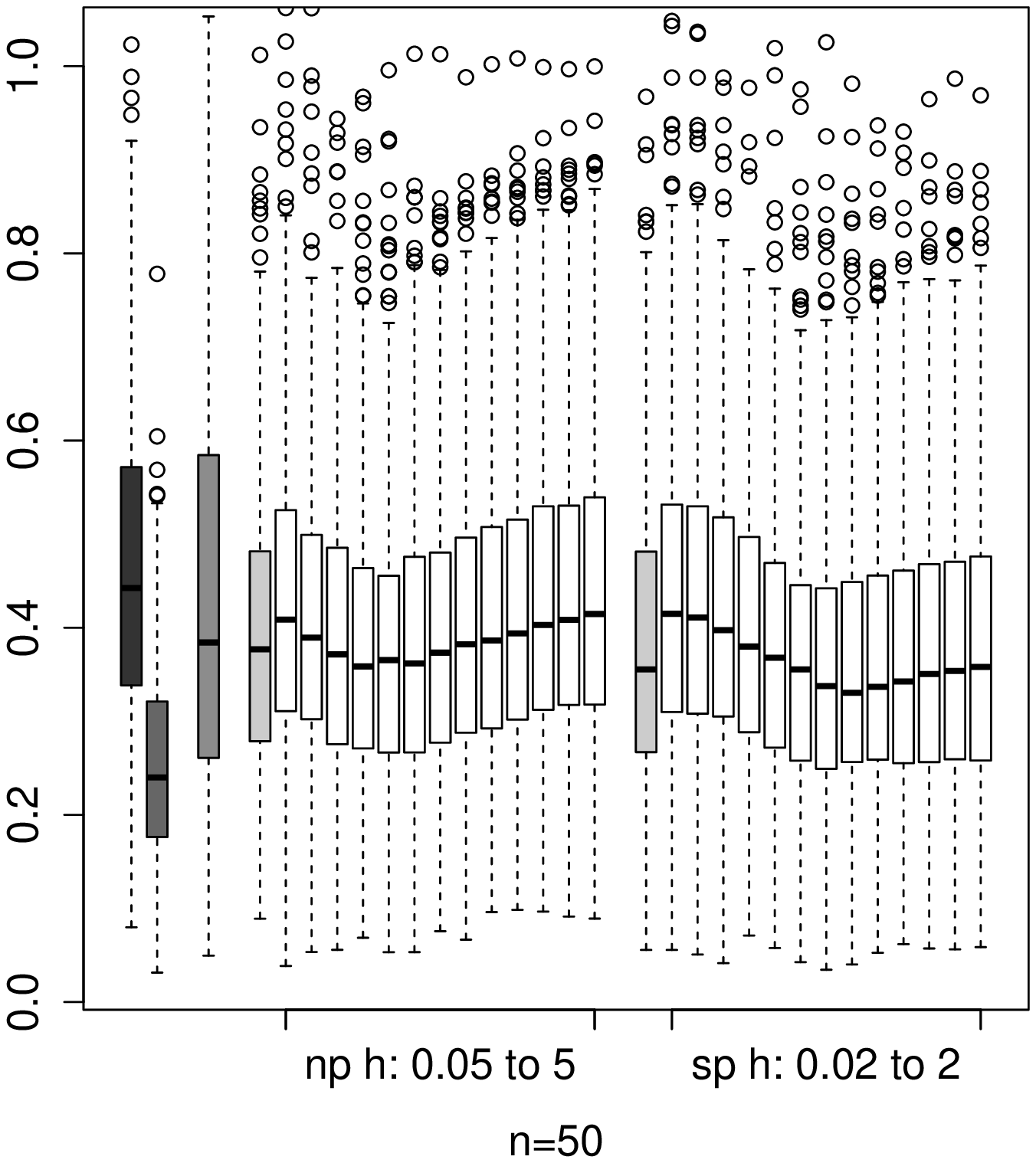} 
\includegraphics[width=5.2cm,height=8.5cm]{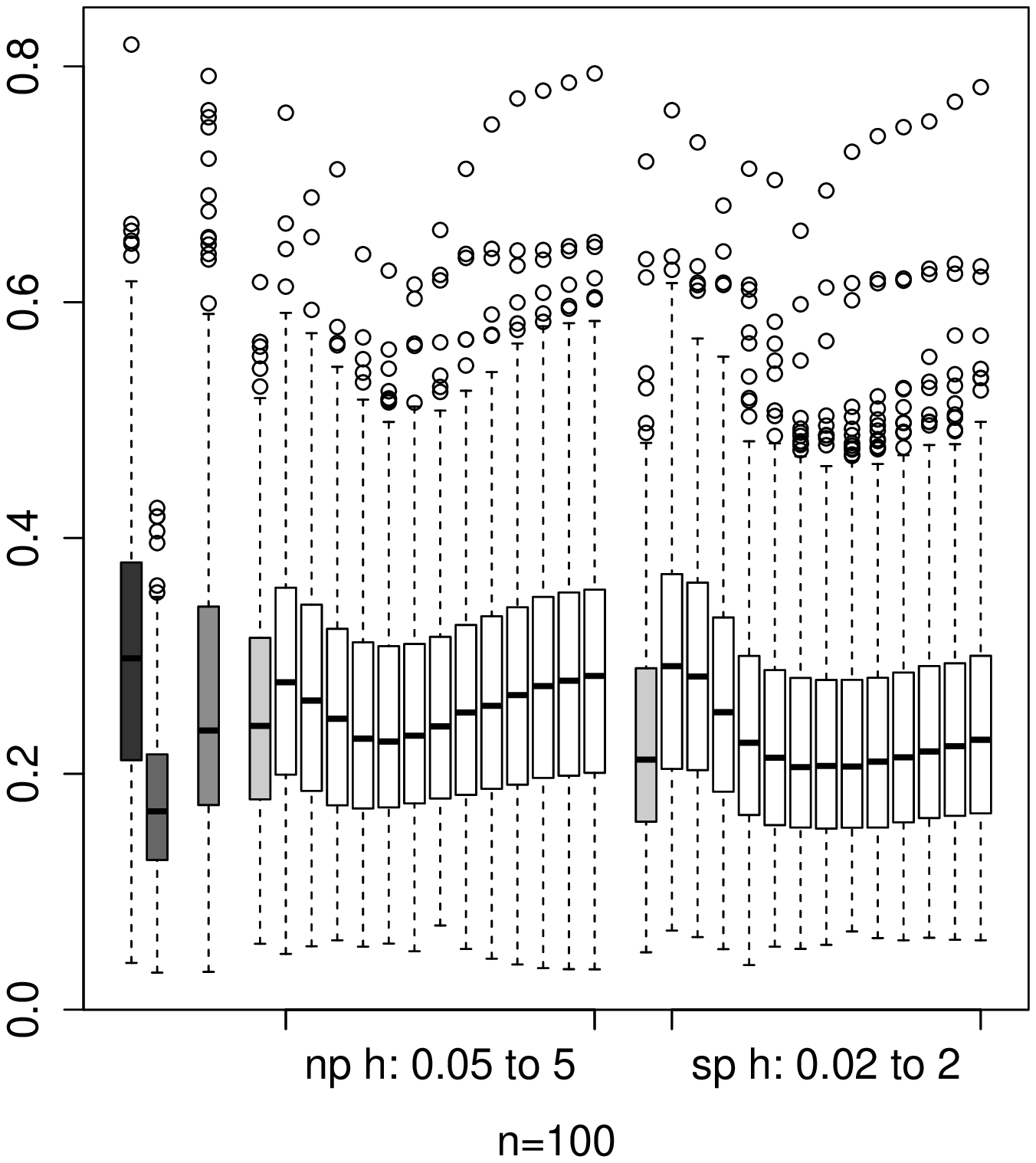} 
\includegraphics[width=5.2cm,height=8.5cm]{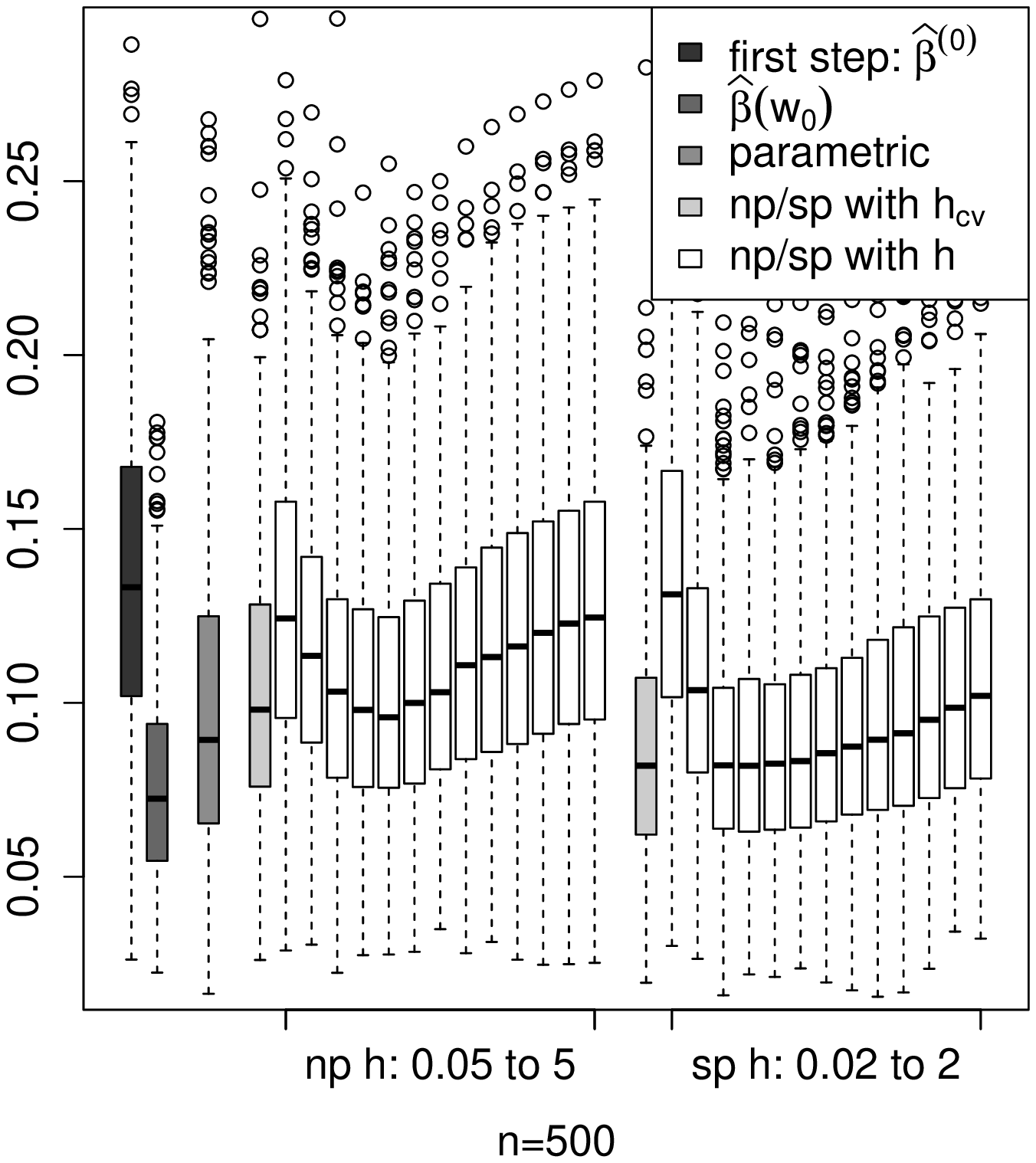} 
\caption{Each boxplot is based on $500$ estimates of $|\w \beta-\beta_0|^2$ when $q=4$ and $\sigma_{0}(x) =\frac 1 2+2\cdot \mathds 1_{\{\beta_{02}^Tx>0\}}$. The parametric, nonparametric (np) and semiparametric (sp) approaches are respectively based on (i), (ii) and (iii).}
\label{simu3}
\end{figure}

\begin{figure}\centering
\includegraphics[width=5.2cm,height=8.5cm]{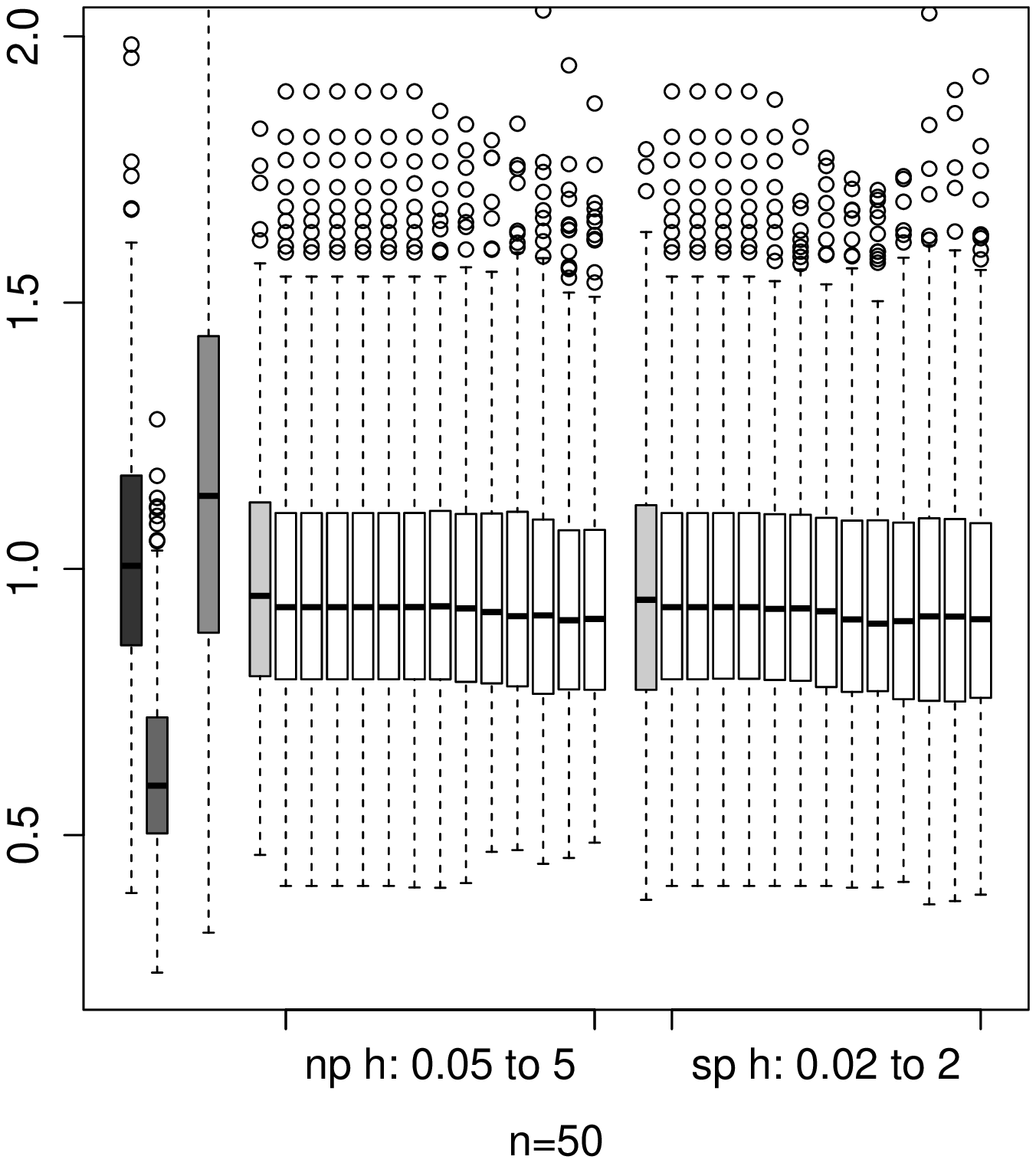}
\includegraphics[width=5.2cm,height=8.5cm]{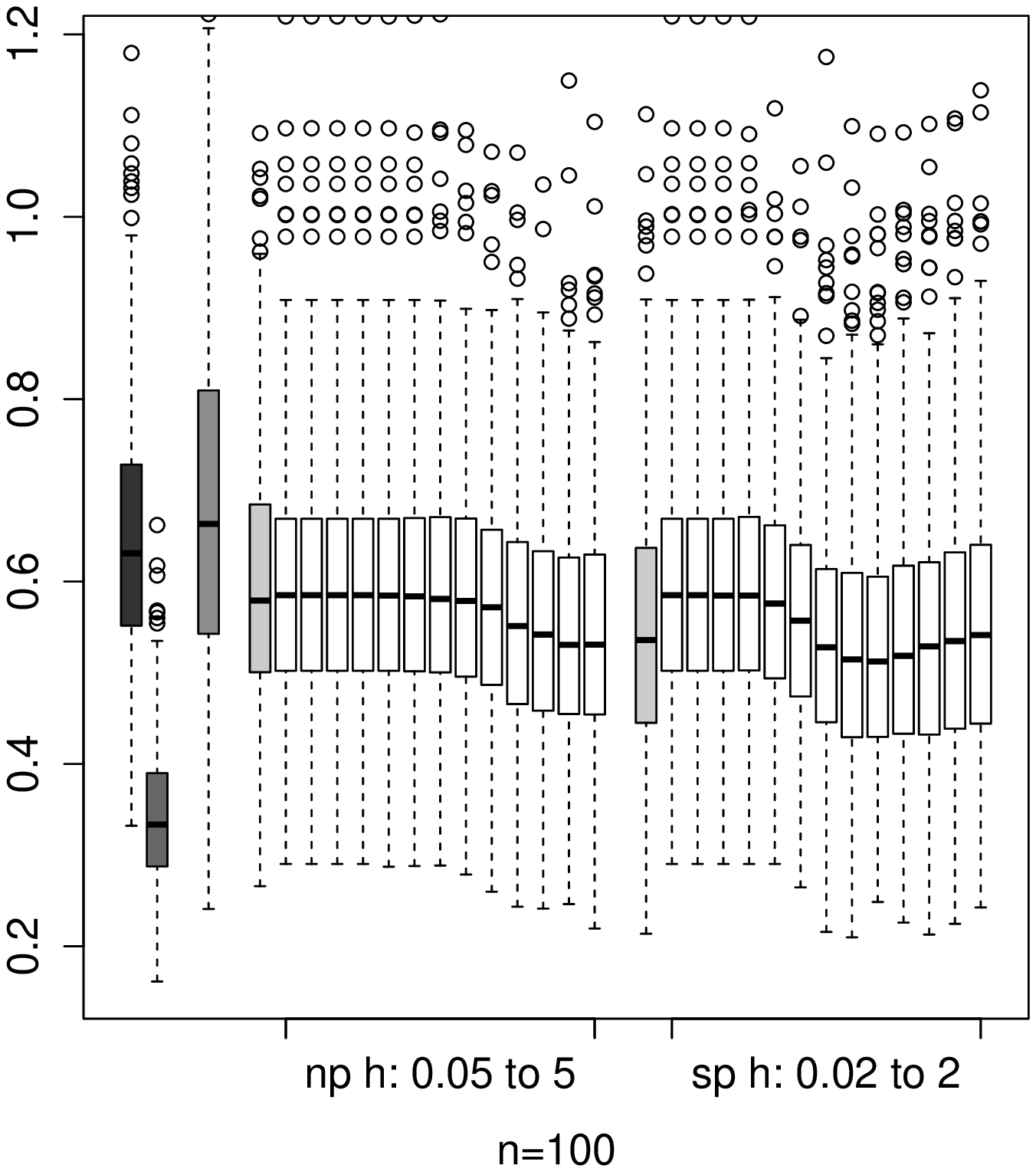} 
\includegraphics[width=5.2cm,height=8.5cm]{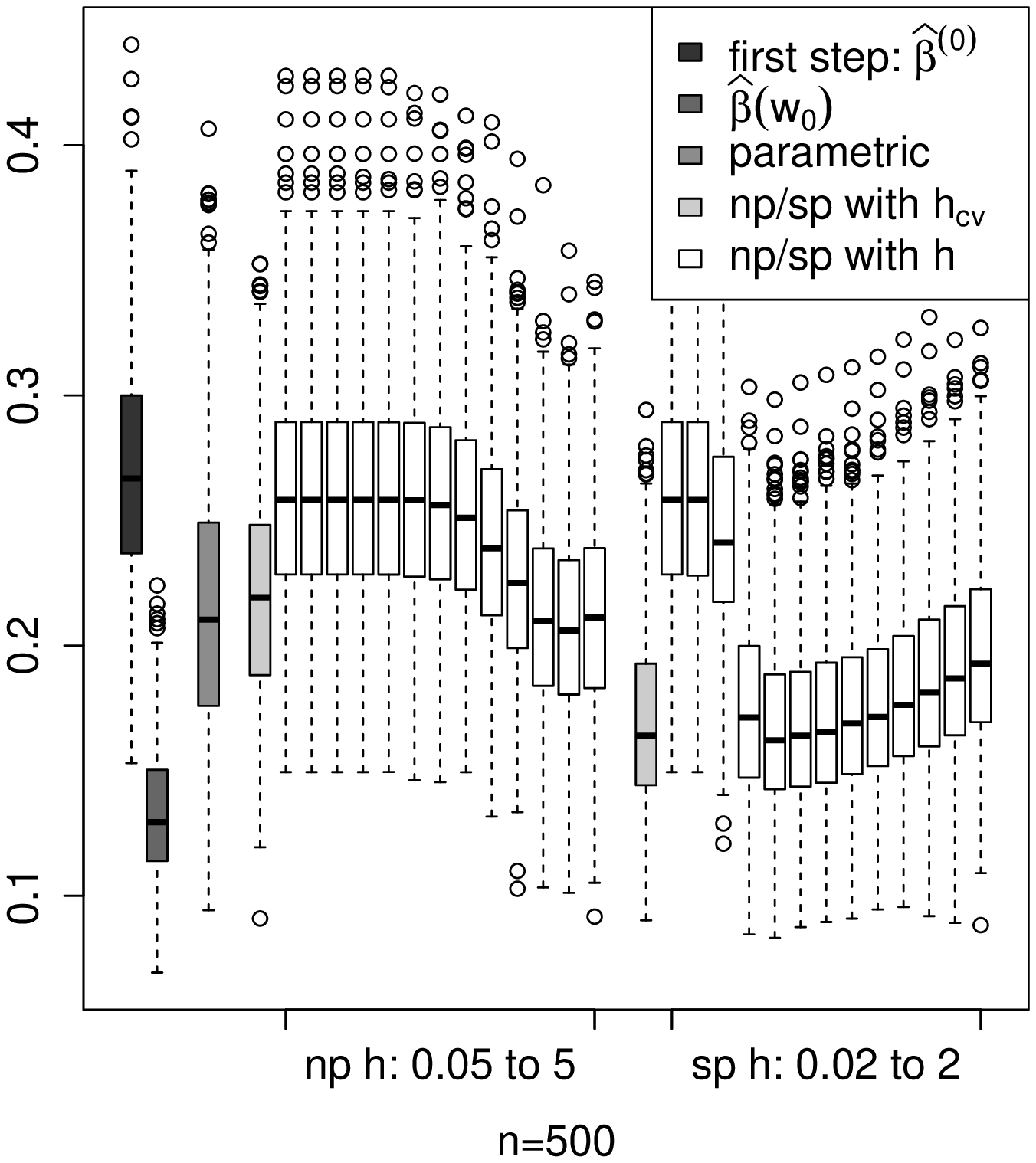} 
\caption{Each boxplot is based on $500$ estimates of $|\w \beta-\beta_0|^2$ when $q=16$ and $\sigma_{0}(x) =\frac 1 2+2\cdot \mathds 1_{\{\beta_{02}^Tx>0\}}$. The parametric, nonparametric (np) and semiparametric (sp) approaches are respectively based on (i), (ii) and (iii).}
\label{simu4}
\end{figure}

The figures \ref{simu1} to \ref{simu4} gives boxplots associated to the estimation error of each method, parametric (i), nonparametric (ii), and semiparamteric (iii), according to different values of $n=50,100,500$, $q=4,16$ and $\sigma_0=\sigma_{01} ,\sigma_{02}$. We also consider the first-step estimator $\w \beta^{(0)}$ and a ``reference estimator" computed with the unknown optimal weights, i.e., $\w \beta(w_0)$. In every case, the accuracy of each method lies between the first step estimator and the reference estimator. In agreement with Theorem \ref{thefficiencyforconditionalmomentrestrictionwithbracketing} and \ref{th:nonparametricclassoffunction}, the gap between the reference estimator and the method (i), (ii), (iii) diminishes as $n$ increases. Each method (i), (ii), (iii), performs differently showing that their equivalence occurs only at very large sample size.

Among the three methods under evaluation (i), (ii), (iii), the clear winner is the semiparametric method (with selection of the bandwidth by cross validation). The fact that it over-rules the nonparametric estimator was somewhat predictable, but the difference in accuracy with the parametric method is surprising. In every situation, the variance and the mean of the error associated to the semiparametric approach are smaller than the variance and the mean of the others. Moreover, the nonparametric method performs as well as the parametric method even in high-dimensional settings. In fact, both approaches are similarly affected by the increase of the dimension. Finally, one sees that the choice of the bandwidth by cross validation works well for both methods nonparametric and semiparametric. In all cases, the estimator with $h_{cv}$ performs similarly to the estimator with the optimal $h$.

\section{Appendix: concentration rates for kernel regression estimators}

 The result follows from the formulation of Talagrand inequality \citep{talagrand1994} given in Theorem 2.1 in \cite{gine2002}.

\begin{lemma}\label{lem:talagrand}
Let $(Y_i\in \R,\ X_i\in \R^q,\ i=1,\ldots, n )$ denote a sequence of random variables independent and identically distributed such that $X_1$ has a bounded density $f$. Given $\widetilde K:\mathbb R^q\mapsto \mathbb R$ such that $\int \widetilde K(u)^2du<+\infty$ and $ \Psi$ a class of real functions defined on $\R^{q+1}$, if both classes $\Psi$ and $\{\widetilde K\left(\frac{x-\cdot}{h}\right),\ x\in \mathbb R^q, \ h>0 \}$ are bounded measurable $VC$ classes, it holds that
\begin{align*}
 \sup_{\psi \in \Psi, \ x\in \mathcal Q} \left|\frac 1 {n}  \sum_{i=1}^n\big\{ \psi(Y_i,X_i) \widetilde K_h(x-X_i)-E[ \psi(Y_1,X_1)  \widetilde K_h(x-X_1)]\big\} \right| =O_{P_\infty}\left ( \sqrt {\frac{|\log(h)|}{nh^{q}}} \right) .
\end{align*}

\end{lemma}

\begin{proof}
 The empirical process to consider is indexed by the product class $\Psi \widetilde{ \mathcal K}_{h}$, where $\widetilde{ \mathcal K}_{h} = \{\widetilde K(\frac{x-\cdot}{h}),\ x\in \mathbb R^q \}$ which is uniformly bounded $VC$ since the product of two uniformly bounded $VC$ classes remains uniformly bounded $VC$. The variance satisfies
\begin{align*}
\text {var}(\psi(Y_1,X_1)  \widetilde K(h^{-1}(x-X_1)) & \leq E \psi(Y_1,X_1)^2  \widetilde K(h^{-1}(x-X_1))^2 \\
&\leq \|\psi\|_\infty^2 \|f\|_\infty  h^q\int \widetilde K(u)^2du ,
\end{align*}
and a uniform bound is given by $\|\psi \widetilde K \|_\infty \leq \sup_{\psi\in \Psi} \| \psi \|_\infty\|\widetilde K \|_\infty$. The application of Theorem 2.1 in \cite{gine2002} gives the specified bound.
\end{proof}

\bibliographystyle{chicago}

\end{document}